\documentclass[reqno]{amsart}

\newtheorem{theorem}{Theorem}[section]

\newtheorem{lemma}[theorem]{Lemma}

\theoremstyle{definition}

\newtheorem{example}[theorem]{Example}

\theoremstyle{plain}

\theoremstyle{remark}
\newtheorem{remark}[theorem]{Remark}

\usepackage{amssymb}
\usepackage{amsmath}
\usepackage{amsthm}
\usepackage{color}
\usepackage{graphicx}
\usepackage{xcolor}
\usepackage{subcaption}
\usepackage{enumerate}
\usepackage[draft]{todonotes}   % notes showed
\usepackage{tabularx}

\newcommand{\ri}{\mathrm{i}}

\title[Global Well-Posedness GPJ]{On the Global Well-Posedness of the Inviscid Generalized Proudman--Johnson Equation Using Flow map Arguments}

\author{Florian Kogelbauer}
\address{Institute for Mechanical Systems, ETH Z\"{u}rich, Leonhardstrasse 21, 8092 Z\"{u}rich, Switzerland}
\email{floriank@ethz.ch}
\begin{document}
\maketitle
\today
\begin{abstract}
	We reformulate the Generalized Proudman--Johnson (GPJ) equation with parameter $a$ in Lagrangian variables, where it takes the form of an inhomogeneous Liouville equation. This allows us to provide an explicit formula for the flow map, up to the solution of an ODE. Depending on the parameter $a$, we prove new criteria for global existence or formation of a finite-time singularity and re-derive results from the literature. In particular, we show that there exist smooth initial data which become singular in finite time for $a>1$. We also give a physical derivation of the GPJ equation for general parameter values of $a$.
\end{abstract}

\section{Introduction}
\subsection{General Overview}
We will be concerned with the one-parameter family of partial differential equations
\begin{equation}\label{GPJa}
\begin{cases}
&u_{txx}+uu_{xxx}-au_xu_{xx}=0,\\
&\left.u\right|_{t=0}=u^0,
\end{cases}
\end{equation}
the \textit{generalized Proudman--Johnson equation} (henceforth abbreviated as GPJ equation), for an unknown, scalar function $u$ with initial condition $u^0$ and a parameter $a\in\mathbb{R}$.
We will study equation \eqref{GPJa} together with either \textit{Dirichlet boundary conditions}, i.e., the function $u:[0,1]\times [0,t^*)\to\mathbb{R}$ satisfies
\begin{equation}\label{DirichletBC}
u(0,t)=u(1,t)=0,\qquad t\in [0,t^*),
\end{equation}
or together with \textit{periodic boundary conditions} (normalizing $u$ to have zero mean), i.e., the function $u:\mathbb{R}\times[0,t^*)\to\mathbb{R}$ satisfies
\begin{equation}\label{periodicBC}
\begin{split}
&u(x+1,t)=u(x,t),\qquad (x,t)\in\mathbb{R}\times [0,t^*),\\
&\int_0^1u(x,t)\, dx=0,\qquad t\in[0,t^*).
\end{split} 
\end{equation}
Throughout, we denote the maximal existence time of a solution to equation \eqref{GPJa} as $t^*$.\\
Equation \eqref{GPJa} was introduced in \cite{okamoto2000} for specific parameter values of $a$ as a model derived from high-dimensional Navier--Stokes equations with certain symmetries. Mainly, however, equation \eqref{GPJa} has been introduced as a mathematical extension to the inviscid Proudman--Johnson equation ($a=1$), cf. \cite{childress_ierley_spiegel_young_1989}, which models inviscid, incompressible fluids close to a wall.\\
Indeed, defining the mock vorticity as $\omega:=u_{xx}$, equation \eqref{GPJa} becomes
\begin{equation}
\omega_t+u\omega_x=a\omega u_x,
\end{equation}
which can be interpreted as an $a$-weighted scalar toy-model for the three-dimensional, inviscid, incompressible vorticity equation, cf. \cite{chorin2012mathematical}. For different values of the parameter $a$, either the transport term $u\omega_x$ or the stretching term $\omega u_x$ becomes dominant. Therefore, understanding global existence and blow-up scenarios for equation \eqref{GPJa} is expected to shed light on the role and interplay of nonlinear transport and nonlinear stretching terms in global well-posedness of solutions.\\
Local well-posedness of equation \eqref{GPJa} in the periodic regime was established in \cite{okamoto2000} (cf. also \cite{Okamoto2009}):
\begin{theorem}[Theorem 2.1 in \cite{Okamoto2009}]
For any $u_{xx}^0\in L^2(0,1)/\mathbb{R}$ there exists $T>0$  and a solution of \eqref{GPJa}, satisfying periodic boundary conditions, unique in the class
\begin{equation}
u_{xx}\in C\Big([0,T];L^2(0,1)/\mathbb{R}\Big)\cap C^1_w\Big([0,T];H^{-1}(0,1)/\mathbb{R}\Big),
\end{equation}
where the subscript $w$ implies weak topology. If, in addition, $u_{xx}^0\in H^m(0,1)/\mathbb{R}$ with $m=1,2,...,$ then $u_{xx}\in C^0\Big([0,T];H^m(0,1)/\mathbb{R}\Big)$.
\end{theorem}
Different global-existence and blow-up scenarios for equation \eqref{GPJa} have been analyzed. The following theorem gives an overview of some criteria for periodic solutions.

\begin{theorem}[Theorem 3.3 in \cite{Okamoto2009}]\label{Okamototheorem}
	The following statements hold true:
\begin{enumerate}
	\item Suppose that $a<-2$ and that $\int_0^1\Big(u_x^0(s)\Big)^3\,ds<0$. Then $\|u_x(t)\|_{L^2(0,1)}$ blows up in finite time.
	\item Suppose that $-2\leq a<-1$. Then $u_x$ remains bounded in $L^2$-norm but blows up in finite time in $L^\infty$ norm unless $u\equiv 0$.
	\item Suppose that $-1\leq a<0$ and that $u_{xx}^0\in L^{-\frac{1}{a}}(0,1)/\mathbb{R}$. Then the solution exists globally in time.
	\item Suppose that $0\leq a<1$ and that $u_{xxx}^0\in L^{\frac{1}{1-a}}(0,1)/\mathbb{R}$. Then the solution exists globally in time.
\end{enumerate}
\end{theorem}

In \cite{CHO2010392}, existence of global weak solutions to equation \eqref{GPJa} for the parameter range $a=-\frac{n+3}{n+1},\quad n\in\mathbb{N}$, and in \cite{cho2012global} for a general $a\in [-2,-1)$ has been proved by the method of characteristics, devised in \cite{bressan2005global} in the context of the Hunter--Saxton equation.\\
In \cite{Sarria2013,SarriaSaxton2013} different global existence and blow-up scenarios based on a representation formula for the derivative of the velocity field along trajectories where proved for $u_x$. In \cite{Sarria2013,SarriaSaxton2013} it is first noted that $u_x$ along trajectories satisfies a Riccati differential equation, which is equivalent to a second-order ODE for which then, in turn, knowledge of a special solution implies a representation formula for the general solution.\\
In \cite{ConstantinWunsch2009,Wunsch2011} blow-up conditions for $a=1$ and general $a$ were given by a method based on the time evolution of suprema and infima. For the special case $a=1$ (Proudman--Johnson equation), different, more specific blow-up criteria can be obtained, cf. \cite{childress_ierley_spiegel_young_1989} for various methods related to trajectories.\\
For parameters $a>1$, comparably little is known about the formation of a singularity. In \cite{Okamoto2009} special, non-smooth self-similar solutions of the form $u(x,t)=\frac{F(x)}{T-t}$ with blow-up time $T$ where constructed by solving the ODE
\begin{equation}\label{ODeselfsimilar}
F''+FF'''-aF'F''=0,
\end{equation}
for $a>1$.  For more special blow-up solutions, we also refer to \cite{nagayama2002blow}. Equation \eqref{ODeselfsimilar}, however, only admits non-smooth solution for $a>1$, while numerical computations suggest that also smooth initial conditions become singular in finite time for this parameter range, cf. \cite{Okamotopresentation}.

\subsection{Results of the paper}
In this paper, we prove new criteria on global existence and formation of a finite-time singularity of equation \eqref{GPJa} in the class of $x$-Lipschitz continuous velocity fields. We remark that the presented methods also apply to lower regularity solutions.\\
First, the GPJ equation is reformulated in Lagrangian variables (Section \ref{Reform}), where it takes the form of an inhomogeneous Liouville equation. Subsequently, along the same lines as in Liouville's original paper \cite{liouville1853equation}, we solve the equation for the flow map and present an explicit solution formula - explicit up to the solution of an ordinary differential equation (Section \ref{SolutionFormula}). This allows us to infer several global existence and singularity criteria (Section \ref{Criteria}). Specifically, we can report the following improvements and novelties compared to the results in \cite{okamoto2000}, i.e., the assumptions of Theorem \ref{mainthm} compare to the assumptions of our main theorem, Theorem \ref{Okamototheorem}, as follows.
\begin{enumerate}
		\item  The assumptions in Theorem \ref{mainthm} for the parameter regime $a<-3$ only involve integrability constraints of the second derivative of $u^0$ or a comparison of the minimum and the $L^2$-norm of $u_x^0$, which are weaker as compared to the sign assumption on $\int_0^1\Big(u_x^0(s)\Big)^3\, ds$ in Theorem \ref{Okamototheorem}.\\
		
		\item The regime where a singularity forms for any sufficiently smooth initial condition is extended from $[-2,-1)$ in Theorem \ref{Okamototheorem} to $[-3,-1)$ in Theorem \ref{mainthm}.\\
		
		%\item The integrability assumptions on the second derivative are weaker as compared to $(3)$ in Theorem \ref{Okamototheorem} for $-1<a<0$.
		%, while for $0<a<1$, Theorem \ref{mainthm} only makes assumptions on the second derivative, but not on the third derivative as $(4)$ in Theorem \ref{Okamototheorem}.Also, global existence for $a=-1$ is proved without any further assumptions on $u^0$.\\
		
		\item Theorem \ref{mainthm} provides assumptions that guarantee global existence for $a>-1$ and assumptions that guarantee the formation of a finite-time singularity for $a\geq 1$, even for smooth initial conditions, for which, to the knowledge of the author, no results have been established so far.\\
\end{enumerate}
 The main theorems are illustrated on several examples. A special emphasis is put on solutions with a parabolic initial velocity field, for which the derivative of the flow map satisfies the constant curvature Liouville equation (Section \ref{constantcurvaturechapter}).\\
 Finally, in Section \ref{Appendix}, we derive the GPJ equation for general parameter values $a\in\mathbb{R}$ from the inviscid, compressible Euler equations.

%Integrating \eqref{GPJxx} with respect to $x$, we obtain\begin{equation}u_{xt}+uu_{xx}+\beta u_x^2=c(t),\end{equation}for some function $t\mapsto c(t)$ that has to be determined from the boundary conditions. 
	%\dedicatory{}

\subsection{Reformulation in Lagrangian Variables}\label{Reform}
In this section, we will reformulate equation \eqref{GPJa} in Lagrangian variables and, subsequently, solve it explicitly in terms of particle trajectories. \\
First, we define the flow associated to the one-dimensional, time dependent vector field $u$, solving equation \eqref{GPJa}, as
\begin{equation}\label{flowu}
\begin{cases}
&F_t(\xi,t)=u(F(\xi,t).t)=:\hat{u}(\xi,t),\quad (\xi,t)\in[0,1]\times [0,t^*),\\
&F(\xi,0)=\xi, \quad \xi\in [0,1].
\end{cases}
\end{equation}
Here, $t^*$ denotes again the maximal existence time for the flow map. We will assume that the initial condition $u^0$ is sufficiently smooth to guarantee existence and uniqueness of solutions to the ODE \eqref{flowu}. In particular, we will assume local well-posedness in the space of $C^1$-functions with Lipschitz continuous derivative, i.e.,
\begin{equation}
u\in C\Big(0,t^*; C^{1,Lip}(0,1)\Big).
\end{equation} 
This allows us to infer that the flow map will have the same regularity as the velocity field in $x$ and, in particular, we can take second derivatives with respect to $x$ almost everywhere. Clearly, also less restrictive assumptions on the velocity field $u$ guarantee the subsequent formulas to hold true.\\
We denote the first and second $x$-derivatives of $u$ along trajectories as
\begin{equation}
\hat{u}_x(\xi,t):=u_x(F(\xi,t),t),\qquad \hat{u}_{xx}(\xi,t):=u_{xx}(F(\xi,t),t),
\end{equation}
where $F$ is a solution to \eqref{flowu}.
Given that the vector field $u$ is $C^1$ with Lipschitz continuous derivative  in $x$ and continuous in $t$, it follows from standard existence and uniqueness theory of ordinary differential equations that the map $F$ defines a local $C^1$-diffeomorphism at any time $t\in [0,\infty)$, which implies that
\begin{equation}\label{posFxi}
F_\xi(\xi,t)> 0, \quad (\xi,t)\in[0,1]\times [0,t^*).
\end{equation}
Indeed, a sign change of $F_\xi$ at some time $t^*$ indicates a break-down of a solution to equation \eqref{flowu}.\\
Assuming Dirichlet boundary conditions \eqref{DirichletBC} the flow map $F$ associated to $u$ has to fixed-points at $\xi=0$ and $\xi=1$, i.e.,
\begin{equation}\label{DirichletF}
F(0,t)=0,\qquad F(1,t)=1,\qquad t\in[0,t^*).
\end{equation}
On the other hand, assuming mean-free periodic boundary conditions \eqref{periodicBC}, the flow map $F(.,t):\mathbb{R}\to\mathbb{R}$ defines a diffeomorphism for every $t\in[0,t^*)$ and we obtain that
\begin{equation}
F(\xi,t)=F(\xi+1,t)-1,\qquad (\xi,t)\in \mathbb{R}\times[0,\infty),
\end{equation} 
thanks to the periodicity of $u$ and the uniqueness of the flow map, or, to put it differently, the map $G(\xi,t):=F(\xi,t)-\xi$ is periodic for $(\xi,t)\in\mathbb{R}\times [0,t^*)$. The mean-free condition on $u$ translates to the following constraint for the flow map $F$:
\begin{equation}\label{flowperiodic}
\begin{split}
0&=\int_0^1u(x,t)\, dx=\int_{F^{-1}(0,t)}^{F^{-1}(0,t)+1}u(F(\xi,t),t)F_\xi(\xi,t)\,d\xi\\
&=\int_{F^{-1}(0,t)}^{F^{-1}(0,t)+1}F_t(\xi,t)F_\xi(\xi,t)\, d\xi,\qquad t\in[0,t^*),
\end{split}
\end{equation}
where we have used the definition of $F$ in \eqref{flowu},  the non-degeneracy condition \eqref{posFxi} and the flow property of $F$.\\

By taking a $\xi$-derivatives in \eqref{flowu}, we obtain expressions for the $x$-derivatives of $u$ along trajectories as
\begin{equation}\label{hatux}
\hat{u}_x(\xi,t)=\frac{F_{t\xi}(\xi,t)}{F_{\xi}(\xi,t)}=\frac{\partial}{\partial t}\log F_\xi(\xi,t),
\end{equation}
where we have used \eqref{posFxi} in the definition of the logarithmic derivative. Taking another $\xi$-derivative in \eqref{hatux}, we obtain
\begin{equation}\label{hatuxx}
\hat{u}_{xx}(\xi,t)=\frac{1}{F_{\xi}(\xi,t)}\left(\frac{F_{t\xi}(\xi,t)}{F_{\xi}(\xi,t)}\right)_{\xi}=\frac{1}{F_{\xi}(\xi,t)}\frac{\partial^2}{\partial t\partial \xi}\log F_\xi(\xi,t).
\end{equation}
The time derivative of $\hat{u}_{xx}$ along trajectories is given by
\begin{equation}\label{hatuxxt}
\begin{split}
(\hat{u}_{xx})_t(\xi,t)&=u_{xxt}(F(\xi,t),t)+\hat{u}_{xxx}(\xi,t)F_t(\xi,t)\\
&=u_{xxt}(F(\xi,t),t)+\hat{u}_{xxx}(\xi,t)\hat{u}(\xi,t).
\end{split}
\end{equation}
Evaluating \eqref{GPJa} along $x=F(\xi,t)$ and using \eqref{hatuxxt}, we find that
\begin{equation}
\frac{d}{dt}\hat{u}_{xx}=a\hat{u}_x\hat{u}_{xx},
\end{equation}
which can be solved explicitly to
\begin{equation}\label{solvehatuxx}
\hat{u}_{xx}(\xi,t)=u_{xx}^0(\xi)\exp\left\{a\int_0^t\hat{u}_x(\xi,s)\, ds\right\},
\end{equation}
where we have used that $\hat{u}_{xx}(\xi,0)=u_{xx}(F(\xi,0),0)=u_{xx}(\xi,0)=u_{xx}^0(\xi)$. At this point, we note that formula \eqref{solvehatuxx} for $a=1$ appears in \cite{ConstantinWunsch2009} with details of how it was derived.\\
Inserting the expressions for $\hat{u}_{x}$ and $\hat{u}_{xx}$ in terms of the flow map $F$ and its derivatives as derived in \eqref{hatux} and \eqref{hatuxx}, equation \eqref{solvehatuxx} becomes
\begin{equation}\label{reformF}
\begin{split}
\frac{1}{F_{\xi}(\xi,t)}\frac{\partial^2}{\partial t\partial \xi}\log F_\xi(\xi,t)&=u_{xx}^0(\xi)\exp\left\{a\int_0^t\frac{d}{ds}\log F_\xi(\xi,s)\, ds\right\}\\
&=u_{xx}^0(\xi)F_\xi(\xi,t)^a,
\end{split}
\end{equation}
for $(\xi,t)\in [0,1]\times [0,t^*)$, where we have used that $\log F_{\xi}(\xi,0)=\log(1)=0$, by the definition of $F$ in \eqref{flowu}. Multiplying equation \eqref{reformF} by $F_\xi$ and introducing
\begin{equation}
f_a(\xi,t)=F_\xi(\xi,t)^{a+1},
\end{equation}
for $a\in\mathbb{R}\setminus\{-1\}$, we arrive at the equation
\begin{equation}\label{Liouvilleinhom}
\frac{\partial^2}{\partial t\partial \xi}\log f_a(\xi,t)=(a+1)u_{xx}^0(\xi)f_a(\xi,t),\qquad (a\neq-1)
\end{equation}
for $(\xi,t)\in [0,1]\times [0,t^*)$ and $a\neq -1$, while for $a=-1$, we simply obtain
\begin{equation}\label{a=-1}
\frac{\partial^2}{\partial t\partial \xi}\log F_\xi(\xi,t)=u_{xx}^0(\xi),\qquad (a=-1),
\end{equation}
for $(\xi,t)\in [0,1]\times [0,t^*)$.\\
\begin{remark}
Equation \eqref{Liouvilleinhom} is an inhomogeneous Liouville equation with  $(a+1)u^0_{xx}(\xi)$ as a curvature term. For $(a+1)u^0_{xx}(\xi)=const.$,  the solution $f_a$ describes the conformal factor of a metric on a surface with constant curvature, cf. \cite{liouville1853equation}. In Section \ref{constantcurvaturechapter}, we solve the constant curvature equation for two values of $a$ explicitly, exemplifying this remarkable property of the GPJ equation.   
\end{remark}

\subsection{Solution formula for the flow map $F$}\label{SolutionFormula}
First, we will start with the explicit solution to equation \eqref{a=-1}. Integrating with respect to $x$ and $t$ in \eqref{a=-1}, we find that
\begin{equation}
\log F_\xi(\xi,t)=\Big[u_x^0(\xi)-u_x^0(0)\Big]t+\theta(t)+\rho(\xi),
\end{equation}
for two functions $\rho:[0,1]\to\mathbb{R}$ and $\theta:[0,t^*)\to\mathbb{R}$. Denoting $\theta_0:=\theta(0)$, it follows from the definition of $F$ in \eqref{flowu} that $0=\log (1)=\log F_\xi(\xi,0)=\theta_0+\rho(\xi)$ and hence, after taking an exponential, 
\begin{equation}
F_\xi(\xi,t)=\exp\left\{\theta(t)-\theta_0+\Big[u_x^0(\xi)-u_x^0(0)\Big]t\right\}.
\end{equation}
Integrating once more with respect to $\xi$ we arrive at the expression
\begin{equation}
F(\xi,t)=\int_0^\xi \exp\left\{\theta(t)-\theta_0+\Big[u_x^0(l)-u_x^0(0)\Big]t\right\}\, dl+\mu(t),
\end{equation}
for some function $\mu:[0,t^*)\to\mathbb{R}$.\\
Assuming Dirichlet boundary conditions, it follows that $\mu(t)\equiv0$, since $F(0,t)=0$ for all $t\in[0,t^*)$, and also that

\begin{equation}
e^{\theta(t)-\theta_0}=\left[\int_0^1 \exp\left\{\Big[u_x^0(l)-u_x^0(0)\Big]t\right\}\, dl\right]^{-1},
\end{equation}

since $F(1,t)=1$ for all $t\in[0,t^*)$. The flow map then takes the from
\begin{equation}
F(\xi,t)=\frac{\int_0^\xi \exp\left\{\Big[u_x^0(l)-u_x^0(0)\Big]t\right\}\, dl}{\int_0^1 \exp\left\{\Big[u_x^0(l)-u_x^0(0)\Big]t\right\}\, dl},
\end{equation}
or, after cancellation of the term $e^{-u^0_x(0)t}$,
\begin{equation}\label{flowa=0}
F(\xi,t)=\frac{\int_0^\xi \exp\left\{u_x^0(l)t\right\}\, dl}{\int_0^1 \exp\left\{u_x^0(l)t\right\}\, dl}.
\end{equation}
\medskip

We now turn to the case $a\neq -1$. To facilitate notation and to emphasize the connection to the geometric nature of equation \eqref{Liouvilleinhom}, we set
\begin{equation}
K_a(\xi)=(a+1)u_{xx}^0(\xi),\qquad  \xi\in[0,1]. 
\end{equation}
In our analysis, we follow the original approach of Liouville as outlined in \cite{liouville1853equation} (in the cited paper, however, only the constant curvature case $K_a=const.$ is considered).\\
Let $f_a=\frac{\partial g_a}{\partial t}$, for some function $g_a:[0,1]\times [0,t^*)\to\mathbb{R}$. As $K_a$ is $t$-independent, equation \eqref{Liouvilleinhom} can be integrated with respect to $t$:
\begin{equation}\label{solveLiouville1}
\frac{\partial }{\partial\xi}\log\frac{\partial g_a}{\partial t}=\left(\frac{\partial g_a}{\partial t}\right)^{-1}\frac{\partial^2 g_a}{\partial t\partial \xi}=K_ag_a+\rho_1,
\end{equation}
for some function $\xi\mapsto\rho_1(\xi), \rho_1:[0,1]\to\mathbb{R}$. Multiplying equation \eqref{solveLiouville1} with $\frac{\partial g_a}{\partial t}$ and integrating once more with respect to $t$, we have that
\begin{equation}\label{solveLiouville2}
\frac{\partial g_a}{\partial \xi}=\frac{K_a}{2}g_a^2+\rho_1g_a+\rho_2,
\end{equation}
for some function $\xi\mapsto\rho_2(\xi), \rho_2:[0,1]\to\mathbb{R}$.\\
 Equation \eqref{solveLiouville2} is a Riccati differential equation in the $\xi$-variable and can be solved by quadrature, once a particular solution is known. So, let $\xi\mapsto h_a(\xi), h:[0,1]\to\mathbb{R}$ be a particular (t-independent) solution to equation \eqref{solveLiouville2} and define
 \begin{equation}\label{solveLiouville4}
 g_a(\xi,t)=h_a(\xi)+\frac{1}{L_a(\xi,t)},\qquad (\xi,t)\in [0,1]\times [0,t^*).
 \end{equation}
 Then $L_a$ satisfies the linear ordinary differential equation
 \begin{equation}\label{solveLiouville3}
 \begin{split}
 \frac{\partial L_a}{\partial\xi}&=-(K_ah_a+\rho_1)L_a-\frac{K_a}{2}\\
 &=:-\rho_3L_a-\frac{K_a}{2}.
 \end{split}
 \end{equation}
 Equation \eqref{solveLiouville3} can be integrated explicitly to
 \begin{equation}
 \begin{split}
 L_a(\xi,t)&=\exp\left\{-\int_0^\xi\rho_3(s)\,ds\right\}\left(\phi(t)-\int_0^\xi\frac{K_a(s)}{2}\exp\left\{\int_0^s\rho_3(l)\,dl\right\}\, ds\right)\\
 &=:\frac{1}{R(\xi)}\left(\phi(t)-\int_0^\xi\frac{K_a(s)}{2}R(s)\, ds\right),
 \end{split}
 \end{equation}
for some function $\phi:[0,t^*)\mapsto \mathbb{R}$ and some function $R:[0,1]\to(0,t^*)$. Setting $\psi(\xi):=-\frac{1}{2}\int_0^\xi K_a(s)R(s)\, ds$, we can express \eqref{solveLiouville4} as
\begin{equation}
g_a(\xi,t)=h_a(\xi)-\frac{2\psi'(\xi)}{K_a(\xi)[\phi(t)+\psi(\xi)]},\qquad (\xi,t)\in[0.1]\times[0,t^*),
\end{equation}
and hence, $f_a$ becomes
\begin{equation}
   f_a(\xi,t)=\frac{2\psi'(\xi)\phi'(t)}{K_a(\xi)[\phi(t)+\psi(\xi)]^2},\qquad (\xi,t)\in[0.1]\times[0,t^*),
\end{equation}
for two arbitrary functions $\phi:[0,t^*)\to\mathbb{R}$ and $\psi:[0,1]\to\mathbb{R}$. Since $f_a> 0$, we have to choose the functions $\phi$ and $\psi$ such that
\begin{equation}\label{condposphipsi}
\frac{\psi'(\xi)\phi'(t)}{K_a(\xi)}> 0,\qquad (\xi,t)\in [0,1]\times [0,t^*).
\end{equation}
In particular, $\phi$ is either increasing or decreasing. Finally, since $f_a>0$, we arrive at
\begin{equation}\label{solveLiouville5}
F_\xi(\xi,t)=\left(\frac{2\psi'(\xi)\phi'(t)}{K_a(\xi)[\phi(t)+\psi(\xi)]^2}\right)^\frac{1}{a+1},\qquad (\xi,t)\in[0.1]\times[0,t^*).
\end{equation}
We will determine the functions $\phi$ and $\psi$ in \eqref{solveLiouville5} to match the boundary conditions of the flow map $F$. First, evaluating \eqref{solveLiouville5} at $t=0$ and using the definition of $F$ in \eqref{flowu}, it follows that
\begin{equation}\label{solvepsi}
\frac{2\psi'(\xi)\phi'_0}{K_a(\xi)[\phi_0+\psi(\xi)]^2}=1,
\end{equation}
where we have denoted $\phi(0)=\phi_0$ and $\phi'(0)=\phi'_0$. Equation \eqref{solvepsi} can be integrated explicitly to
\begin{equation}\label{solvepsi2}
\psi(\xi)=-\frac{1+\phi_0\left(\overline{\kappa}+\frac{1}{2\phi'_0}\int_0^\xi K_a(s)\,ds\right)}{\overline{\kappa}+\frac{1}{2\phi'_0}\int_0^\xi K_a(s)\,ds},\qquad \xi\in[0,1],
\end{equation}
for some constant $\overline{\kappa}\in\mathbb{R}$. Since $K_a(\xi)=(a+1)u_{xx}^0(\xi)$, equation \eqref{solvepsi2} reduces to
\begin{equation}
\psi(\xi)=-\frac{1+\phi_0\left(\kappa+\frac{a+1}{2\phi'_0}u^0_x(\xi)\right)}{\kappa+\frac{a+1}{2\phi'_0}u^0_x(\xi)},\qquad \xi\in[0,1],
\end{equation}
where we have denoted $\kappa=\overline{\kappa}-\frac{a+1}{2\phi_0'}u_x^0(0)$, and $F_\xi$ becomes
\begin{equation}\label{solveLiouville6}
F_\xi(\xi,t)=\left(\frac{\phi'(t)}{\phi_0'}\right)^\frac{1}{a+1}\left[1-(\phi(t)-\phi_0)\left(\kappa+\frac{a+1}{2\phi'_0}u^0_x(\xi)\right)\right]^{-\frac{2}{a+1}},
\end{equation}
for $(\xi,t)\in [0,1]\times [0,t^*).$
Rescaling
\begin{equation}
\phi(t)\mapsto\frac{1}{\phi_0'}(\phi(t)-\phi_0),
\end{equation} 
we can assume that $\phi_0=0$ as well as $\phi_0'=1$ and hence, equation \eqref{solveLiouville6} simplifies to 
\begin{equation}\label{Fxiphi}
F_\xi(\xi,t)=\left(\phi'(t)\right)^\frac{1}{a+1}\left[1-\phi(t)\left(\kappa+\frac{a+1}{2}u^0_x(\xi)\right)\right]^{-\frac{2}{a+1}},\qquad  (\xi,t)\in [0,1]\times [0,t^*),
\end{equation}
for some $\kappa\in\mathbb{R}$. We can give an expression of the constant $\kappa$ in terms of the function $\phi$ as follows.  Taking a $t$-derivative in equation \eqref{Fxiphi}, it follows from the initial condition $F_{t\xi}(\xi,0)=u_{x}^0(\xi)$, $\xi\in[0,1]$, that
\begin{equation}
\frac{1}{a+1}\phi_0''+\frac{2}{a+1}\left(\kappa+\frac{a+1}{2}u_x^0(\xi)\right)=u_x^0(\xi),\quad \xi\in[0,1],
\end{equation}
where we have used that $\phi(0)=0$, as well as $\phi'(0)=1$ and where we have abbreviated $\phi''_0:=\phi''(0)$. Hence, $\kappa$ is given as
\begin{equation}\label{kappaphipp}
\kappa=-\frac{\phi_0''}{2}.
\end{equation}
Sine $\kappa$ as given by equation \eqref{kappaphipp} is a second-derivative property of the function $\phi$, the flow map $F$ should not depend upon $\kappa$. Indeed, we can define another function $\eta$ through the M\"obius transform
\begin{equation}
\eta(t)=\frac{\phi(t)}{1-\kappa\phi(t)},\qquad \phi(t)=\frac{\eta(t)}{1+\kappa\eta(t)},\qquad t\in (0,t^*),
\end{equation}
which also satisfies $\eta(0)=0$ and $\eta'(0)=1$ such that the derivative of the flow map becomes
\begin{equation}\label{Fxi}
F_\xi(\xi,t)=\left(\eta'(t)\right)^\frac{1}{a+1}\left[1-\eta(t)\frac{a+1}{2}u^0_x(\xi)\right]^{-\frac{2}{a+1}},\qquad  (\xi,t)\in [0,1]\times [0,t^*),
\end{equation}
i.e., we have eliminated the dependence upon $\kappa$ while keeping the first-derivative properties of the unknown function $\eta$.\\
Integrating equation \eqref{Fxi} with respect to $\xi$ gives
\begin{equation}\label{flowmapDirichlet}
F(\xi,t)=\mu(t)+\left(\eta'(t)\right)^\frac{1}{a+1}\int_0^\xi\left[1-\eta(t)\frac{a+1}{2}u^0_x(s)\right]^{-\frac{2}{a+1}}\, ds,\qquad  (\xi,t)\in [0,1]\times [0,t^*),
\end{equation}
for some function $\mu:[0,t^*)\to\mathbb{R}$.\\
Assuming Dirichlet boundary conditions \eqref{DirichletF}, it immediately follows that $\mu\equiv 0$ and hence
\begin{equation}\label{FDirichlet}
F(\xi,t)=\left(\eta'(t)\right)^\frac{1}{a+1}\int_0^\xi\left[1-\eta(t)\frac{a+1}{2}u^0_x(s)\right]^{-\frac{2}{a+1}}\,ds,\qquad  (\xi,t)\in [0,1]\times [0,t^*).
\end{equation}
To match the boundary condition $F(1,t)=1,\quad t\in[0,t^*)$, the function $\eta$ has to satisfy the implicit ordinary differential equation
\begin{equation}\label{eqeta}
\left(\eta'(t)\right)^\frac{1}{a+1}\int_0^1\left[1-\eta(t)\frac{a+1}{2}u^0_x(s)\right]^{-\frac{2}{a+1}}\,ds=1,\qquad  t\in [0,t^*),
\end{equation}
which, in general, cannot be solved explicitly.\\
On the other hand, assuming periodic mean-free boundary conditions \eqref{periodicBC}, which imply that $F(\xi,t)=F(\xi+1,t)-1,\quad t\in[0,t^*)$, we obtain that $\eta$ has to satisfy equation \eqref{eqeta} as well. By the mean-free constraint \eqref{flowperiodic}, the function $\mu$ is then given as the solution to the differential equation
\begin{footnotesize}
\begin{equation}
\begin{split}
\mu'(t)=-\int_{F^{-1}(0,t)}^{F^{-1}(0,t)+1}\left(\eta'(t)\right)^\frac{1}{a+1}&\left[1-\eta(t)\frac{a+1}{2}u^0_x(\xi)\right]^{-\frac{2}{a+1}}\times\\
&\times\frac{d}{dt}\left(\left(\eta'(t)\right)^\frac{1}{a+1}\int_0^\xi\left[1-\eta(t)\frac{a+1}{2}u^0_x(s)\right]^{-\frac{2}{a+1}}\, ds\right)\, d\xi,
\end{split}
\end{equation}
\end{footnotesize}

\noindent
with initial condition $\mu(0)=0$, since $F(0,\xi)=\xi$.\\
Note that formula  \eqref{flowmapDirichlet} only holds as long as $\eta'(t)\geq 0$ and $1-\eta(t)\frac{a+1}{2}u^0_x(\xi)>0$ for all $\xi\in(0,1)$. 

\begin{example}[$a=-3$, Burger's equation]
Assume Dirichlet boundary conditions and let $a=-3$. For this parameter value of $a$, the GPJ equation reduces to Burger's equation, for which the flow map can be computed easily as
\begin{equation}
F_{Burgers}(\xi,t)=u^0(\xi)t+\xi,\qquad (\xi,t)\in [0,1]\times [0,t^*).
\end{equation}
Now, consider the flow map \eqref{flowmapDirichlet} with $a=-3$:
\begin{equation}\label{FBurgers}
\begin{split}
F(\xi,t)&=\sqrt{\frac{1}{\eta'(t)}}\int_0^\xi 1+\eta(t)u^0_x(s)\, ds\\
&=\sqrt{\frac{1}{\eta'(t)}}\Big(\xi +\eta(t)u^0(\xi) \Big),
\end{split}
\end{equation}
where we have used that $u^0(0)=0$ in the second step. Note that formula \eqref{FBurgers} holds only for $t\in[0,t^*)$, where $1+\eta(t)u^0_x(s)> 0$. From the second boundary condition, $F(1,t)=1, t\in[0,t^*)$ and $u^0(1)=0$, it follows that $\eta$ satisfies the ordinary differential equation
\begin{equation}\label{eqphiBurgers}
\sqrt{\frac{1}{\eta'(t)}}=1,\qquad t\in[0,t^*],
\end{equation}
which, remembering that $\eta(0)=0$ as well as $\eta'(0)=1$, implies that
\begin{equation}
\eta(t)=t.
\end{equation}
%Equation \eqref{eqphiBurgers} can be integrated explicitly to\begin{equation}\label{phiBurgers}\eta(t)=\frac{\kappa(t+\tilde{\kappa})-1}{\kappa^2(t+\tilde{\kappa})},\qquad t\in[0,t^*),\end{equation}for some $\tilde{\kappa}\in\mathbb{R}$. Evaluating \eqref{phiBurgers} at $t=0$ gives\begin{equation}\eta(0)=0=\frac{\kappa\tilde{\kappa}-1}{\kappa^2\tilde{\kappa}},\end{equation}and hence\begin{equation}\tilde{\kappa}=\frac{1}{\kappa}.\end{equation}Note that $\eta'(0)=1$ is satisfied automatically.
The flow map then takes the form
\begin{equation}
F(\xi,t)=u^0(\xi)t+\xi,\quad (\xi,t)\in [0,1]\times [0,t^*),
\end{equation}
which is exactly the flow map of Burger's equation. In particular, we recover the explicit formula for the blow-up time,
\begin{equation}
t^*=-\frac{1}{\min_{\xi\in[0,1]}u_x^0(\xi)}.
\end{equation}
\end{example}

\begin{example}[$a=-2$, Hunter--Saxton equation]
Assume Dirichlet boundary conditions and let $a=-2$. For this parameter value of $a$, the GPJ equation reduces to the Hunter--Saxton equation.
Consider now the flow map \eqref{flowmapDirichlet} with $a=-2$:
\begin{equation}\label{flowHS}
\begin{split}
F(\xi,t)&=\frac{1}{\eta'(t)}\int_0^\xi\left[1+\eta(t)\frac{1}{2}u^0_x(\xi)\right]^2\, ds\\
&=\frac{1}{\eta'(t)}\left[\xi+ u^0(\xi)\eta(t)+\frac{\eta^2(t)}{4}\int_0^\xi(u_x^0(s))^2\,ds\right], 
\end{split}
\end{equation}
for $(\xi,t)\in [0,1]\times [0,t^*)$. We remark again that the representation formula \eqref{flowHS} only holds as long as
\begin{equation}\label{condHS}
1+\eta(t)\frac{1}{2}u^0_x(\xi)> 0.
\end{equation} 
Evaluating equation \eqref{flowHS} at $\xi=1$ and using the boundary condition $F(1,t)=1, t\in[0,t^*)$, we obtain the following ordinary differential equation for $\eta$:
\begin{equation}\label{etaeqHS}
\eta'(t)=1+\frac{1}{4}\|u_x^0\|_{L^2(0,1)}^2\eta^2(t),\qquad t\in[0,t^*].
\end{equation}
%Since the discriminant of the quadratic polynomial in equation \eqref{etaeqHS} is given by $\Delta=-\|u_x^0\|_{L^2(0,1)}^2<0$,
Equation \eqref{etaeqHS} can be integrated to 
\begin{equation}
\eta(t)=\frac{2}{\|u_x^0\|_{L^2(0,1)}}\tan\left(\frac{t}{2}\|u_x^0\|_{L^2(0,1)}\right),\qquad  t\in[0,t^*],
\end{equation}
where we have used that $\eta(0)=0$ and $\eta'(0)=1$.
%for some $\tilde{\kappa}\in\mathbb{R}$. Since $\eta(0)=0$, we find that\begin{equation}\tilde{\kappa}=-\arctan\left(\frac{2\kappa}{\|u_x^0\|_{L^2(0,1)}}\right),\end{equation}and hence\begin{equation}\label{etatan}\begin{split}\eta(t)&=\frac{4\kappa+2\|u_x^0\|_{L^2(0,1)}\tan\left(\frac{t}{2}\|u_x^0\|_{L^2(0,1)}-\arctan\left(\frac{2\kappa}{\|u_x^0\|_{L^2(0,1)}}\right)\right)}{4\kappa^2+\|u_x^0\|_{L^2(0,1)}^2}\\&=\frac{4\kappa+2\|u_x^0\|_{L^2(0,1)}\left(\frac{\tan\left(\frac{t}{2}\|u_x^0\|_{L^2(0,1)}\right)-\frac{2\kappa}{\|u_x^0\|_{L^2(0,1)}^2}}{1+\frac{2\kappa}{\|u_x^0\|_{L^2(0,1)}^2}\tan\left(\frac{t}{2}\|u_x^0\|_{L^2(0,1)}\right)}\right)}{4\kappa^2+\|\dfrac{num}{den}u_x^0\|_{L^2(0,1)}^2}\\&=\frac{2}{\|u_x^0\|_{L^2(0,1)}}\tan\left(\frac{t}{2}\|u_x^0\|_{L^2(0,1)}\right),\qquad  t\in[0,t^*],\end{split}\end{equation}where we have used the addition theorem for tangent. From equation \eqref{etatan}, it immediately follows that $\eta'(0)=1$ as well as $\eta''_0(0)=0$, and hence, by formula \eqref{kappaphipp}, $\kappa=0$.\\\todo{The flow map is independent upon kappa, so it could be that already phi is independent upon kappa, or only F is independent upon kappa}
Finally, the flow map takes the form
\begin{equation}
\begin{split}
F(\xi,t)&=\cos^{2}\left(\frac{t}{2}\|u_x^0\|_{L^2(0,1)}\right)\left[\xi+\frac{2}{\|u_x^0\|_{L^2(0,1)}}\tan\left(\frac{t}{2}\|u_x^0\|_{L^2(0,1)}\right)u^0(\xi)\right.\\
&\qquad\left.+\frac{1}{\|u_x^0\|_{L^2(0,1)}^2}\tan^2\left(\frac{t}{2}\|u_x^0\|_{L^2(0,1)}\right)\int_0^\xi(u_x^0(s))^2\,ds\right],\quad (\xi,t)\in[0,1]\times [0,t^*).
\end{split}
\end{equation}
Since $\eta$ blows up in finite time, we find that condition \eqref{condHS} will be violated at the blow-up time
\begin{equation}
t^*=-\frac{2}{\|u_x^0\|_{L^2(0,1)}}\arctan\left(\frac{\|u_x^0\|_{L^2(0,1)}}{\min_{\xi\in[0,1]}u^0_x(\xi)}\right).
\end{equation}
\end{example}

\subsection{Constant Curvature Solutions}\label{constantcurvaturechapter}

Consider the GPJ equation with Dirichlet boundary conditions and initial condition $u^0(x)=\gamma x(x-1)$, cf. Figure \ref{Parabola}.

\begin{figure}
	\includegraphics[scale=0.8]{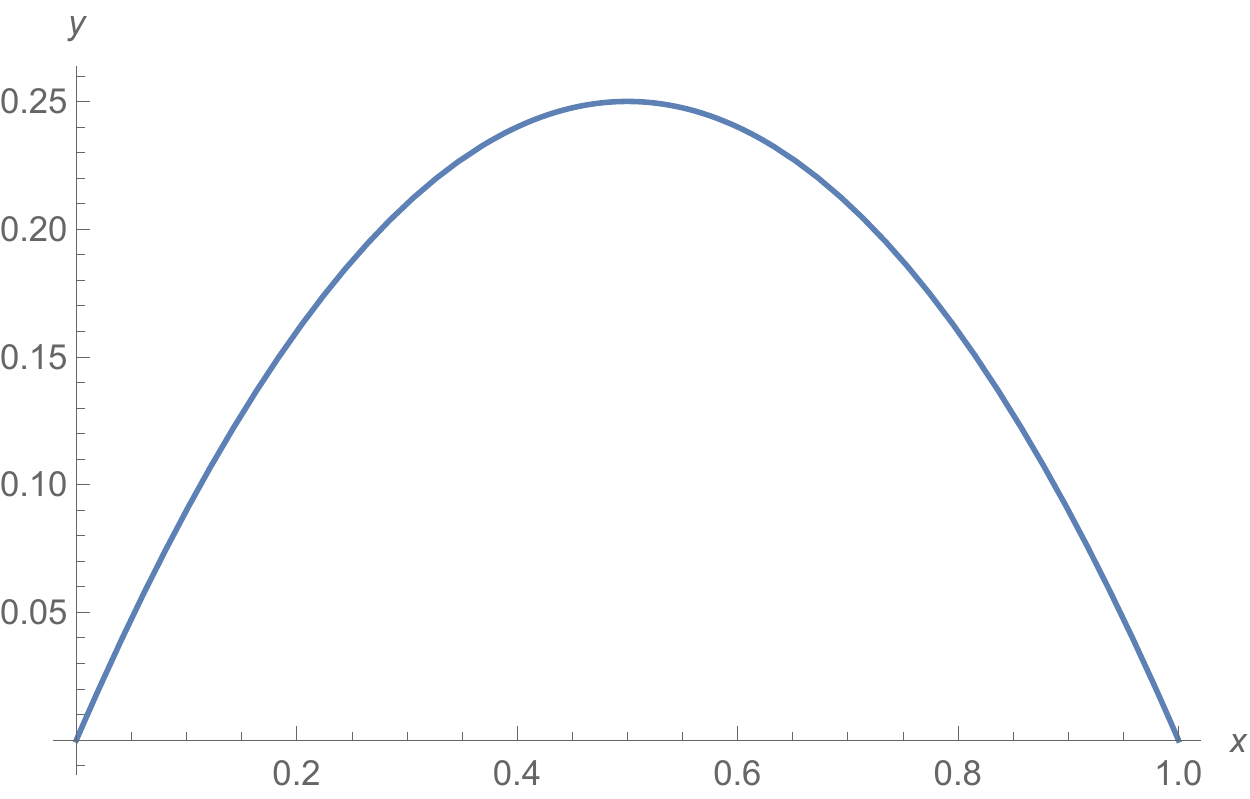}	
	\caption{The initial condition $u^0(x)=-x(x-1)$, satisfying Dirichlet boundary conditions.} 
	\label{Parabola}
\end{figure}

 Since $u_x^0(x)=\gamma (2x-1)$, the representation formula \eqref{flowmapDirichlet} gives
\begin{equation}\label{flowconstant}
\begin{split}
F(\xi,t)&=(\eta'(t))^\frac{1}{a+1}\int_0^\xi\Big[1-\gamma\frac{a+1}{2}(2s-1)\eta(t)\Big]^{-\frac{2}{a+1}}\, ds\\
&=(\eta'(t))^\frac{1}{a+1}\frac{a+1}{-2A(t)(a-1)}\left([1+(1-2\xi)A(t)]^{\frac{a-1}{a+1}}-[1+A(t)]^\frac{a-1}{a+1}\right),
\end{split}
\end{equation}
for $(\xi,t)\in[0,1]\times [0,t^*)$ and $a\neq 1$. (The case $a=1$ will be treated below.)
 where we have abbreviated
\begin{equation}
A(t)=\gamma\frac{a+1}{2}\eta(t).
\end{equation}
Using the boundary condition $F(1,t)=1,\quad t\in[0,t^*)$, the formula \eqref{flowconstant} leads to the following ordinary differential equation for the function $\eta$:
\begin{equation}\label{ODEconstantcurvature}
1=(\eta'(t))^\frac{1}{a+1}\frac{1}{\gamma(1-a)\eta(t)}\left([1-\gamma\frac{a+1}{2}\eta(t)]^{\frac{a-1}{a+1}}-[1+\gamma\frac{a+1}{2}\eta(t)]^\frac{a-1}{a+1}\right),
\end{equation}
for $t\in (0,t^*)$, which, for general $a$, cannot be solved in terms of elementary functions.\\

\begin{example}[Constant Curvature for $a=0$]
Consider the constant curvature flow \eqref{flowconstant} for $a=0$. Equation \eqref{ODEconstantcurvature} becomes
\begin{equation}
1=\frac{\eta'(t)}{1-\frac{\gamma^2}{4}\eta^2(t)},\quad t\in(0,t^*),
\end{equation}
which, using $\eta_0=0$ and $\eta'_0=1$, can be solved to $\eta(t)=\frac{2}{\gamma}\tanh\left(\frac{\gamma}{2} t\right),\quad t\in(0,t^*)$. Hence, the flow map takes the form
\begin{equation}\label{Flowa1}
F(\xi,t)=\frac{\xi\left(\coth\left(\frac{\gamma t}{2}\right)-1\right)}{1-2\xi+\coth\left(\frac{\gamma t}{2}\right)},\quad (\xi,t)\in [0,1]\times (0,\infty).
\end{equation}
In particular, the flow map exists for all times. The time evolution of the $\xi$ of \eqref{Flowa1} now describes a surface of constant curvature, cf. Figure \ref{3Dconstantcurvature}.
\begin{figure}
	\includegraphics[scale=0.5]{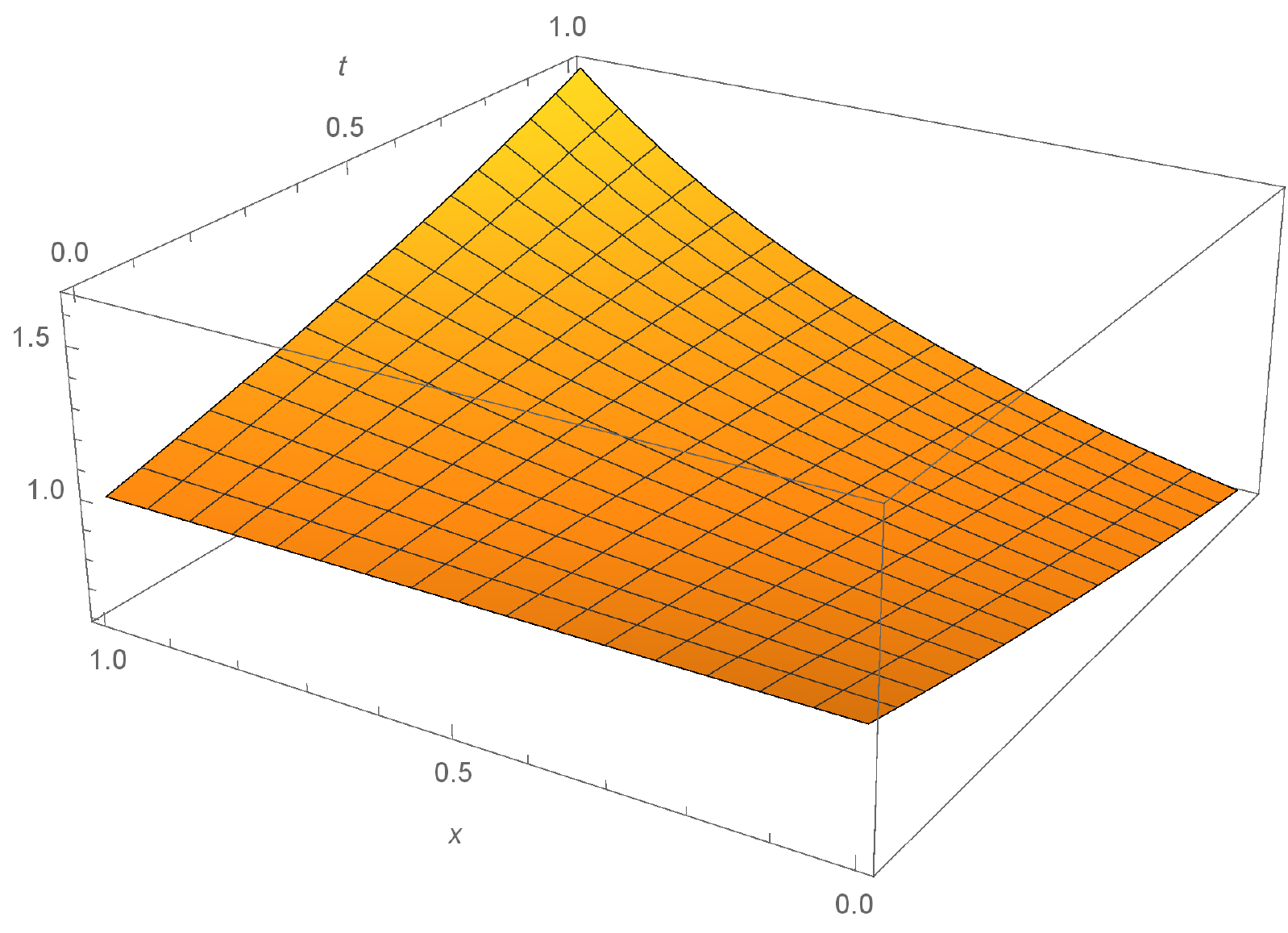}	
	\caption{The $\xi$-derivative of \eqref{Flowa1} for the parameter value $\gamma=2$ and for $(\xi,t)\in[0,1]\times [0,1]$.} 
	\label{3Dconstantcurvature}
\end{figure}

\end{example}

\begin{example}[Constant  Curvature for $a=1$]\label{examplea=1}
	For $a=1$, equation \eqref{GPJa} becomes the Proudman--Johnson equation. We revisit the example $u^0(x)=\gamma x(x-1)$, which was already treated in \cite{childress_ierley_spiegel_young_1989}, from a flow map point of view. The ordinary differential equation \eqref{eqeta} takes the form
	\begin{equation}\label{eqlog}
	\begin{split}
	1&=\sqrt{\eta'(t)}\int_0^1\frac{1}{1-\eta(t)\gamma(2s-1)}\, ds\\
	&=-\sqrt{\eta'(t)}\frac{1}{2\gamma\eta(t)}\log\left|\frac{1-\gamma\eta(t)}{1+\gamma\eta(t)}\right|\, ds,
	\end{split}
	\end{equation}
for $t\in[0,t^*)$. By comparing the signs of $\gamma$ and by noting that $\eta(t)$ will be positive for, at least, small $t$, equation \eqref{eqlog} can be integrated implicitly to
\begin{equation}\label{defPsi}
\begin{split}
t&=\frac{1}{4\gamma^2}\int_0^\eta\frac{1}{y^2}\log^2\left|\frac{1-\gamma y}{1+\gamma y}\right|\, dy\\
&=:\Psi(\eta).
\end{split}
\end{equation}
Since the function $y\mapsto \frac{1}{y}\log\left|\frac{1-\gamma y}{1+\gamma y}\right|$ is square integrable over $\mathbb{R}$, as can be seen by integration by parts, the function $\eta\mapsto \Psi(\eta)$ does not blow up and hence, the map $t\mapsto\eta(t)$ is unbounded, in fact, even blows up in finite time cf. Figure \ref{FigPsieta}. This implies, that the condition $1+\eta(t)\frac{1}{2}u^0_x(\xi)> 0$ will be violated at a finite time $t^*<\infty$ and hence, no global solution exists. 
\begin{figure}
	\includegraphics[scale=0.8]{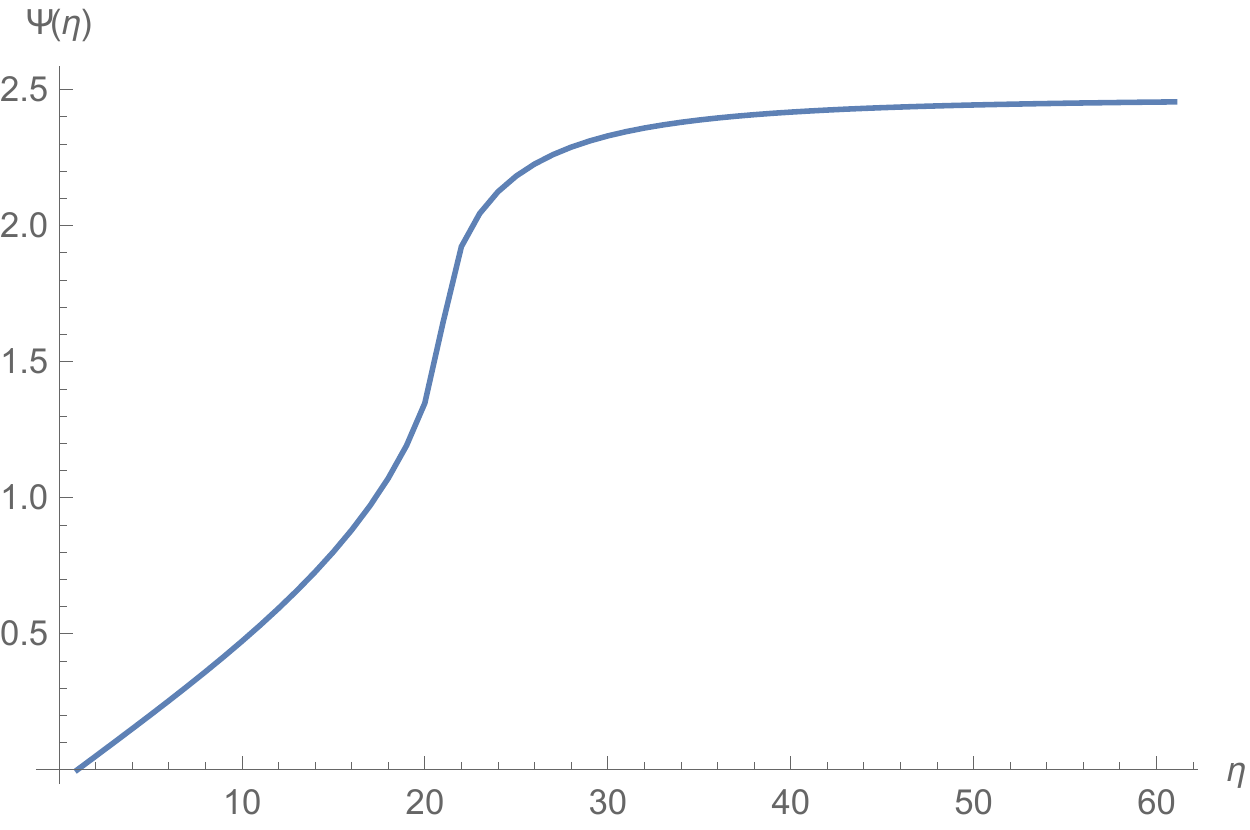}	
	\caption{The function $\eta\mapsto\Psi(\eta)$ as defined in \eqref{defPsi} for the parameter value $\gamma=2$.} 
	\label{FigPsieta}
\end{figure}
\end{example}

\section{Criteria for formation of a singularity and global existence}\label{Criteria}

In this section, we present some singularity criteria and global well-posedness criteria for different values of $a$.\\
We note that the representation formula \eqref{flowmapDirichlet} only holds true as long as $1-\eta(t)\frac{a+1}{2}u_x^0(\xi)>0$. If, for some $(\xi^*,t^*)\in[0,1]\times(0,\infty)$, we have that $0=1-\eta(t^*)\frac{a+1}{2}u_x^0(\xi^*)=\Big(\eta'(t^*)\Big)^{-\frac{1}{a+1}}F_\xi(\xi^*,t^*)$, then the flow map $F$ ceases to be a diffeomorphism of the interval $[0,1]$ and hence, the velocity field $u$ can no longer belong to the class $u\in C\Big(0,t^*; C^{1,Lip}(0,1)\Big)$. Therefore, global existence of solutions $u\in C\Big(0,t^*; C^{1,Lip}(0,1)\Big)$ is equivalent to preservation of the inequality $1-\eta(t)\frac{a+1}{2}u_x^0(\xi)>0$ for all $(\xi,t)\in[0,1]\times [0,\infty)$.\\
Throughout, we write
\begin{equation}
-u_{min}\leq u_x^0(\xi)\leq u_{max},\quad \xi\in[0,1],
\end{equation}
for the negative minimal value and the positive maximal value of $u_x^0$. In the proof of our main theorem, we will need the following inequality.

\begin{lemma}[Reverse quadratic Bernoulli inequality]
	Let $x>-1$ and let $0<\alpha<1$. Then
	 \begin{equation}\label{reverseBernoulli}
	 (1+x)^\alpha\geq \frac{\alpha(\alpha-1)}{2}x^2+\alpha x+\frac{\alpha(3-\alpha)}{2}.
	 \end{equation}
\end{lemma}
\begin{proof}
Since $\alpha-1<0$, we can apply Bernoulli's inequality to obtain
\begin{equation}\label{Bernoulli1}
(1+x)^{\alpha-1}\geq 1+(\alpha-1)x,\quad x>-1.
\end{equation}
Integrating the inequality \eqref{Bernoulli1} from $-1$ to $x$ gives
\begin{equation}
\frac{1}{\alpha}(1+x)^\alpha\geq x+1+\frac{\alpha-1}{2}x^2-\frac{\alpha-1}{2},\quad x>-1,
\end{equation}
which, since $\alpha>0$, gives \eqref{reverseBernoulli}.
\end{proof}

\begin{remark}
	Note that it is only possible to obtain a quadratic inequality of the form \eqref{reverseBernoulli} for $0<\alpha<1$, as for this parameter regime, $\alpha-1<0$ holds, allowing to apply the usual Bernoulli inequality to $(1+x)^{\alpha-1}$, while at the same time $\lim_{x\to-1}(1+x)^\alpha=0$.
\end{remark}

The following theorem gives conditions under which the solution exists for all times or becomes singular in finite time, depending on the parameter $a$. 

\begin{theorem}\label{mainthm}
Let $u\in C\Big(0,t^*; C^{1,Lip}(0,1)\Big)$ be a solution to equation \eqref{GPJa}, either with Dirichlet boundary conditions \eqref{DirichletBC} or with periodic boundary conditions and zero mean \eqref{periodicBC}, and let  $t^*\in (0,\infty]$ be its maximal existence time.
\begin{itemize}
	
	\item For $a<-3$, the solution exists only locally in time, i.e., $t^*<\infty$, provided that
	\begin{equation}\label{Assa-3}
	\|u_x^0\|_{L^2(0,1)}<\sqrt{\frac{3a+5}{a+3}}u_{min},
	\end{equation}
	or provided that
	\begin{equation}\label{Ass-3alternative}
	u_{xx}^0\in L^{\frac{2}{2+(a+1)q}}(0,1),
	\end{equation}
	for some $0<q<-\frac{2}{a+1}$.\\
	
	\item For $-3\leq a<-1$, the solution exists only locally in time, i.e., $t^*<\infty$ and the maximal existence time can be estimated as
	\begin{equation}
	0<t^*<-\frac{2}{(a+1)u_{min}}.
	\end{equation}
	
	\item For $a=-1$, the solution exists globally in time, i.e., $t^*=\infty$.\\
	
	\item For $-1<a<0$, the solution exists globally in time, i.e., $t^*=\infty$, provided that
	%\begin{equation}	\lim_{\eta\to\frac{2}{(a+1)u_{max}}}\int_0^1\frac{ds}{1-\frac{a+1}{2}\eta u_x^0(s)}=\infty.	\end{equation}\\
	\begin{equation}\label{Assuxx}
	u_{xx}^0\in L^{-\frac{1}{a}}(0,1).
	\end{equation}\\

	\item For $-1<a$, the solution exists globally in time, i.e., $t^*=\infty$, provided that
	\begin{equation}\label{Assa-10}
	u_{max}< \frac{1}{\sqrt{1+a}}\|u_x^0\|_{L^2(0,1)},
	\end{equation}
	
	\item For $1\leq a$, the solution exists only locally in time, i.e., $t^*<\infty$, provided that the function
	\begin{equation}\label{defPsiThm}
	\psi(\eta):=\int_0^1\frac{ds}{1-\frac{a+1}{2}\eta u_x^0(s)},
	\end{equation}
	is square-integrable over the interval $[0,\frac{2}{(a+1)u_{max}}]$.
\end{itemize}
\end{theorem}

\begin{remark}
		It is not immediately clear that assumptions \eqref{Assa-10} and the square-integrability of expression \eqref{defPsiThm} are mutually exclusive. We refer to Example \ref{example} for the special case of a piece-wise linear initial condition. For initial conditions which oscillate equally around zero, i.e., initial conditions with $u_{max}=u_{min}$, there is a upper bound on $a$ such that \eqref{Assa-10} can be satisfied. 
		Indeed, by Jensen's inequality
		\begin{equation}
		u_{max}<\frac{1}{\sqrt{1+a}}\|u_x^0\|_{L^2(0,1)}<\frac{1}{\sqrt{1+a}}u_{max},
		\end{equation}
		implying that
		\begin{equation}
		-1<a<0.
		\end{equation}
\end{remark}

\begin{proof}
We divide the proof of Theorem \ref{mainthm} in different section, each one employing different techniques to deduce the claims. To facilitate notation, we introduce the parameter
\begin{equation}
\alpha=-\frac{2}{a+1},
\end{equation}
whose values in dependence upon $a$ are shown in Table \ref{table}.\\
\begin{center}
	\begin{table}[h!]
\begin{tabular}{ |l|l| }
	\hline
	\multicolumn{2}{|c|}{Parameter values $a$ and $\alpha$} \\
	\hline
	$ -\infty<a<-3$ & $\,\,\,\,\,0<\alpha<1$ \\
	$-3\leq a<-1$ & $\,\,\,\,\,1\leq\alpha<\infty$ \\
	$-1<a<0$ & $-\infty<\alpha<-2$ \\
	$\,\,\,\,\,0<a<1$ & $-2<\alpha<-1$\\
	$\quad1<a,\infty$ & $-1<\alpha<0$\\
	\hline
\end{tabular}
\caption{}
\label{table}
\end{table}
\end{center}

\textbf{The case $-\infty<a<-3$ or $0<\alpha<1$}\\
Assume, to the contrary, that $1+\eta (t)\frac{u^0_x(\xi)}{\alpha}>0$, while $\eta'(t)>0$, for all $(\xi,t)\in[0,1]\times[0,\infty)$. First, let us assume \eqref{Assa-3}. This assumption allows us to apply the reversed quadratic Bernoulli inequality \eqref{reverseBernoulli} to equation \eqref{eqeta}, which in turn implies that
\begin{equation}\label{inequa-3}
\begin{split}
1&=\left(\eta'(t)\right)^\frac{1}{a+1}\int_0^1\left[1+\eta(t)\frac{u^0_x(s)}{\alpha}\right]^\alpha\,ds\\
&\geq\left(\eta'(t)\right)^\frac{1}{a+1}\int_0^1\frac{(\alpha-1)}{2\alpha}\Big(u^0_x(s)\eta(t)\Big)^2+u^0_x(s)\eta(t)+\frac{\alpha(3-\alpha)}{2}\, ds\\
&=\left(\eta'(t)\right)^\frac{1}{a+1}\left(\frac{\alpha(3-\alpha)}{2}+\frac{(\alpha-1)}{2\alpha}\|u_x^0\|^2_{L^2(0,1)}\eta^2(t)\right),
\end{split}
\end{equation}
for all $t\in(0,\infty)$, where in the last equality, we have used either Dirichlet or periodic boundary conditions. Since $a+1<0$, inequality \eqref{inequa-3} is equivalent to
\begin{equation}
1\leq \eta'(t)\left(\frac{\alpha(3-\alpha)}{2}+\frac{(\alpha-1)}{2\alpha}\|u_x^0\|^2_{L^2(0,1)}\eta^2(t)\right)^{a+1},
\end{equation}
for $t\in (0,\infty)$, as long as the expression in the brackets is non-negative. Therefore, $\eta$ can be bounded from below by a monotonically increasing function which goes to 
\begin{equation}
\eta^*=\alpha\sqrt{\frac{3-\alpha}{1-\alpha}}\frac{1}{\|u_x^0\|_{L^2(0,1)}},
\end{equation}
as $t\to\infty$. Therefore, for any $\varepsilon>0$, there exists a time $t^*(\varepsilon)$, such that, at the minimum of the expression $1+\eta (t)\frac{u^0_x(\xi)}{\alpha}$ in $\xi$, we find that
\begin{equation}
\begin{split}
1-\eta(t^*)\frac{u_{min}}{\alpha}&\leq 1-\frac{u_{min}}{\alpha}\left(\alpha\sqrt{\frac{3-\alpha}{1-\alpha}}\frac{1}{\|u_x^0\|_{L^2(0,1)}}-\varepsilon\right)\\
&< 1-\sqrt{\frac{3-\alpha}{1-\alpha}}\frac{u_{min}}{\|u_x^0\|_{L^2(0,1)}}+\frac{u_{min}}{\alpha}\varepsilon.
\end{split}
\end{equation}
By assumption \eqref{Assa-3} and since $\varepsilon$ was arbitrary, we obtain a contradiction.\\

Now, let us assume \eqref{Ass-3alternative}. For any $0<q<\alpha$ and any two numbers $\xi_1,\xi_2\in[0,1]$, we can estimate
\begin{equation}\label{inequa-3alternative}
\begin{split}
\frac{\alpha}{\eta(q+1)}\left[\left(1+\eta\frac{u_x^0(\xi)}{\alpha}\right)^{q+1}\right]_{\xi_1}^{\xi_2}&=\frac{\alpha}{\eta(q+1)}\int_{\xi_1}^{\xi_2}\frac{d}{ds}\left(1+\eta\frac{u_x^0(s)}{\alpha}\right)^{q+1}\, ds\\
&=\int_{\xi_1}^{\xi_2}u_{xx}^0(s)\left(1+\eta\frac{u_x^0(s)}{\alpha}\right)^{q}\, ds\\
&\leq \|u_{xx}^0\|_{L^\frac{\alpha}{\alpha-q}}\left\|\left(1+\eta\frac{u^0_{x}}{\alpha}\right)^q\right\|_{L^{\frac{\alpha}{q}}(0,1)}\\
&=\|u_{xx}^0\|_{L^{\frac{\alpha}{\alpha-q}}(0,1)}\left(\int_0^1\left[1+\eta\frac{u^0_{x}(s)}{\alpha}\right]^\alpha\, ds\right)^{\frac{q}{\alpha}},
\end{split}
\end{equation}
for $0<\eta<\frac{|\alpha|}{u_{min}}$, where we have applied H\"older's inequality with the pair of conjugated exponents  $\left(\frac{\alpha}{q},\frac{\alpha}{\alpha-q}\right)$, using assumption \eqref{Ass-3alternative}. Since $u_x^0$ is mean free for both Dirichlet and periodic boundary conditions, we can find $\xi_1<\xi_2\in [0,1]$ such that $u_x^0(\xi_1)<0$, while $u_x^0(\xi_2)>0$ (possibly after multiplying the first equality in \eqref{inequa-3alternative} by $-1$). This implies that we can bound
\begin{equation}\label{inequxi1xi2}
\begin{split}
\frac{\alpha}{\eta(q+1)}\left[\left(1+\eta\frac{u_x^0(\xi)}{\alpha}\right)^{q+1}\right]_{\xi_1}^{\xi_2}&=\frac{\alpha}{\eta(q+1)}\left[\left(1+\eta\frac{u_x^0(\xi_2)}{\alpha}\right)^{q+1}-\left(1+\eta\frac{u_x^0(\xi_1)}{\alpha}\right)^{q+1}\right]\\
&\geq \frac{\alpha}{\eta(q+1)}\left[\left(1+\eta\frac{u_x^0(\xi_2)}{\alpha}\right)^{q+1}-1\right],
\end{split}
\end{equation}
since, for this choice of $\xi_1$, we have that $0<\left(1+\eta\frac{u_x^0(\xi_1)}{\alpha}\right)^{q+1}<1$.
 Since also $\frac{q}{\alpha}>0$, we can therefore estimate
 \begin{equation}
 \begin{split}
1&=(\eta'(t))^{\frac{1}{a+1}}\left(\int_0^1\left[1+\eta(t)\frac{u^0_{x}(s)}{\alpha}\right]^\alpha\, ds\right)\\
&\geq(\eta'(t))^{\frac{1}{a+1}}\left( \frac{\alpha}{\|u_{xx}^0\|_{L^{\frac{\alpha}{\alpha-q}}(0,1)}}\frac{1}{\eta(t)(q+1)}\left[\left(1+\eta(t)\frac{u_x^0(\xi)}{\alpha}\right)^{q+1}\right]_{\xi_1}^{\xi_2}\right)^{\frac{\alpha}{q}}\\
&\geq (\eta'(t))^{\frac{1}{a+1}} \left( \frac{\alpha}{\|u_{xx}^0\|_{L^{\frac{\alpha}{\alpha-q}}(0,1)}}\frac{1}{\eta(t)(q+1)}\left[\left(1+\eta(t)\frac{u_x^0(\xi_2)}{\alpha}\right)^{q+1}-1\right]\right)^{\frac{\alpha}{q}},
\end{split}
 \end{equation}
for $t\in(0,t^*)$, where we have used inequality \eqref{inequa-3alternative} in the first step and inequality \eqref{inequxi1xi2} in the last step. Since now $a+1<0$, it follows that $\eta'$ can be bounded from below as

 \begin{equation}\label{help}
 \begin{split}
 1&\leq \left( \frac{\alpha}{\|u_{xx}^0\|_{L^{\frac{\alpha}{\alpha-q}}(0,1)}}\frac{1}{(q+1)}\right)^{\frac{\alpha}{q}(a+1)}\eta'(t)[\eta(t)]^{-(a+1)\frac{\alpha}{q}}\left[\left(1+\eta(t)\frac{u_x^0(\xi_2)}{\alpha}\right)^{q+1}-1\right]^{(a+1)\frac{\alpha}{q}}\\
 &= \left( \frac{\alpha}{\|u_{xx}^0\|_{L^{\frac{\alpha}{\alpha-q}}(0,1)}}\frac{1}{(q+1)}\right)^{-\frac{2}{q}}\eta'(t)[\eta(t)]^{\frac{2}{q}}\left[\left(1+\eta(t)\frac{u_x^0(\xi_2)}{\alpha}\right)^{q+1}-1\right]^{-\frac{2}{q}},
 \end{split}
 \end{equation}

by the definition of $\alpha$. Integrating both sides of the inequality \eqref{help} with respect to $t$, we have that
 \begin{equation}
t\leq C\int_0^\eta y^{\frac{2}{q}}\left[\left(1+y\frac{u_x^0(\xi_2)}{\alpha}\right)^{q+1}-1\right]^{-\frac{2}{q}}\, dy,
 \end{equation}
for some constant $C>0$. 
We will now show that
 \begin{equation}\label{intcond}
\int_0^{\frac{\alpha}{u_{min}}} y^{\frac{2}{q}}\left[\left(1+y\frac{u_x^0(\xi_2)}{\alpha}\right)^{q+1}-1\right]^{-\frac{2}{q}}\, dy<\infty,
\end{equation}
to obtain a contradiction.
First, we note that the integrand in \eqref{intcond} is bounded  for $\eta\in(0,\infty)$ (since $0<q\leq\alpha <1$):
 \begin{equation}
 \begin{split}
 \lim_{y\to 0}y^{\frac{2}{q}}\left[\left(1+y\frac{u_x^0(\xi_2)}{\alpha}\right)^{q+1}-1\right]^{-\frac{2}{q}}&=\lim_{y\to 0}y^{\frac{2}{q}}\left[\sum_{n=1}^\infty{q+1 \choose n}\left(y\frac{u_x^0(\xi_2)}{\alpha}\right)^n\right]^{-\frac{2}{q}}\\
 &=\lim_{y\to 0}\left[\sum_{n=1}^\infty{q+1 \choose n}\left(\frac{u_x^0(\xi_2)}{\alpha}\right)^ny^{n-1}\right]^{-\frac{2}{q}}\\
 &=\lim_{y\to 0} \left(\frac{u_x^0(\xi_2)}{\alpha}\right)^{-\frac{2}{q}}\left[\sum_{n=0}^\infty{q+1 \choose n+1}\left(y\frac{u_x^0(\xi_2)}{\alpha}\right)^n\right]^{-\frac{2}{q}}\\
 &=\left(\frac{u_x^0(\xi_2)}{\alpha}\right)^{-\frac{2}{q}}(q+1)^{-\frac{2}{q}}<\infty,
 \end{split}
 \end{equation}
 where we have expanded the expression in a binomial series, using that $y$ is small enough. Since $u_x^0(\xi_2)>0$, the integrand in \eqref{intcond}  does not have any pole for $y\geq 0$. Since 
\begin{equation}
y^{\frac{2}{q}}\left[\left(1+y\frac{u_x^0(\xi_2)}{\alpha}\right)^{q+1}-1\right]^{-\frac{2}{q}}=\mathcal{O}(y^{-2}), \quad \text{ for } y\to\infty,
\end{equation}
the inequality \eqref{intcond} is indeed satisfied - a contradiction.
 %In fact, for $1\leq \eta^*y^{q+1}$ and $\eta^*<\left(\frac{u_x^0(\xi_2)}{\alpha}\right)^{q+1}$,  we can estimate \begin{equation}  \begin{split}   \int_{\eta^*}^\eta y^{\frac{2}{q}}\left[\left(1+y\frac{u_x^0(\xi_2)}{\alpha}\right)^{q+1}-1\right]^{-\frac{2}{q}}\, dy&\leq  \int_{\eta^*}^\eta y^{\frac{2}{q}}\left[\left(y\frac{u_x^0(\xi_2)}{\alpha}\right)^{q+1}-1\right]^{-\frac{2}{q}}\, dy \\  &\leq \int_{\eta^*}^\eta y^{\frac{2}{q}}\left[\left(y\frac{u_x^0(\xi_2)}{\alpha}\right)^{q+1}-\eta^*y^{q+1}\right]^{-\frac{2}{q}}\, dy\\  &\leq \int_{\eta^*}^\eta y^{-2}\left[\left(\frac{u_x^0(\xi_2)}{\alpha}\right)^{q+1}-\eta^*\right]^{-\frac{2}{q}}\, dy, \end{split}  \end{equation}  which remains bounded even for $\eta\to\infty$.\\ Therefore, the function $\eta$, as the inverse of a monotonically increasing, bounded function, goes to infinity in finite time and, in particular, there exists a time $t^*$ such that $1-\eta(t^*)\frac{u_{min}}{\alpha}=0$ - a contradiction. 
 This proves the claim.\\

\textbf{The case $-3\leq a<-1$ or $1\leq\alpha<\infty$}\\
Assume again, to the contrary, that $1+\eta (t)\frac{u^0_x(\xi)}{\alpha}>0$, while $\eta'(t)>0$, for all $(\xi,t)\in[0,1]\times[0,\infty)$. For this range or parameter values, the function $x\mapsto x^\alpha$ is convex, while the function $x\mapsto x^\frac{1}{a+1}$ is monotonically decreasing. \footnote{Note that this is the maximal parameter range with these properties.} We can therefore apply Jensen's inequality to \eqref{eqeta} to obtain
\begin{equation}\label{inequa-3-1}
\begin{split}
1&=\left(\eta'(t)\right)^\frac{1}{a+1}\int_0^1\left[1+\eta(t)\frac{u^0_x(s)}{\alpha}\right]^\alpha\, ds\\
&\geq \left(\eta'(t)\right)^\frac{1}{a+1}\left(\int_0^11+\eta(t)\frac{u^0_x(s)}{\alpha}\, ds\right)^\alpha\\
&= \left(\eta'(t)\right)^\frac{1}{a+1},
\end{split}
\end{equation}
for $t\in[0,\infty)$, where in the last step, we have again used either Dirichlet or periodic boundary conditions. Inverting and integrating inequality \eqref{inequa-3-1} gives
\begin{equation}
t\leq \eta(t),\quad t\in[0,\infty),
 \end{equation}
where we have used the initial data for $\eta$ - again a contraction to $1+\eta (t)\frac{u_{min}}{\alpha}>0$ at $t^*=\frac{\alpha}{u_{min}}$.\\

\textbf{The case $a=-1$ or $\alpha=-\infty$}\\
In this case, we have derived the explicit formula \eqref{flowa=0} flow map, which shows that the solution exists for all times.\\

\textbf{The case $-1<a<0$ or $-\infty<\alpha<-2$}\\
For any two numbers $\xi_1,\xi_2\in [0,1]$, $\xi_1<\xi_2$, we can calculate
\begin{equation}\label{inequ-1a1}
\begin{split}
\frac{-\alpha}{\eta}\left[\left(1+\eta\frac{u_x^0(\xi)}{\alpha}\right)^{-1}\right]^{\xi_2}_{\xi_1}&=\frac{-\alpha}{\eta}\int_{\xi_1}^{\xi_2}\frac{d}{ds}\left(1+\eta\frac{u_x^0(s)}{\alpha}\right)^{-1}\, ds\\
&=\int_{\xi_1}^{\xi_2}u_{xx}^0(s)\left(1+\eta\frac{u_x^0(s)}{\alpha}\right)^{-2}\, ds\\
&\leq \|u_{xx}^0\|_{L^{\frac{\alpha}{\alpha+2}}(0,1)}\left\|\left(1+\eta\frac{u^0_{x}}{\alpha}\right)^{-2}\right\|_{L^{-\frac{\alpha}{2}}(0,1)}\\
&=\|u_{xx}^0\|_{L^{\frac{\alpha}{\alpha+2}}(0,1)}\left(\int_0^1\left[1+\eta\frac{u^0_{x}(s)}{\alpha}\right]^\alpha\, ds\right)^{-\frac{2}{\alpha}},
\end{split}
\end{equation}
for $0<\eta<\frac{|\alpha|}{u_{max}}$, where we have applied H\"older's inequality with the pair of conjugated exponents  $\left(-\frac{\alpha}{2},\frac{\alpha}{\alpha+2}\right)$, using assumption \eqref{Assuxx}.\\
Since $a+1>0$, it follows from \eqref{eqeta} that we can bound
\begin{equation}
\begin{split}
1&\geq \eta'\left(\frac{|\alpha|}{\eta\|u_{xx}^0\|_{L^{\frac{\alpha}{\alpha+2}}(0,1)}}\right)^{-\frac{\alpha}{2}(a+1)}\left(\left[\left(1+\eta\frac{u_x^0(\xi)}{\alpha}\right)^{-1}\right]^{\xi_2}_{\xi_1}\right)^{-\frac{\alpha}{2}(a+1)}\\
&\geq \eta'\left(\frac{|\alpha|}{\|u_{xx}^0\|_{L^{\frac{\alpha}{\alpha+2}}(0,1)}}\right)\left(\left[\left(1+\eta\frac{u_x^0(\xi)}{\alpha}\right)^{-1}\right]^{\xi_2}_{\xi_1}\right)\eta^{-1},
\end{split}
\end{equation}
for any choice of $\xi_1,\xi_2\in [0,1]$, by the definition of $\alpha$. Choosing $\xi_1,\xi_2$ such that 
\begin{equation}\label{xi1xi2}
0<-\frac{\alpha}{u_x^0(\xi_2)}<-\frac{\alpha}{u_x^0(\xi_1)},
\end{equation}
or, equivalently, since $\alpha<0$,
\begin{equation}\label{xi1xi2explicit}
0<u_x^0(\xi_1)<u_x^0(\xi_2),
\end{equation}
we can bound $t\mapsto \eta(t)$ from above by a positive multiple of the inverse of the monotonically increasing function
\begin{equation}
\Psi(\eta):=\int_0^\eta y^{-1}\left(\left[\left(1+y\frac{u_x^0(\xi)}{\alpha}\right)^{-1}\right]^{\xi_2}_{\xi_1}\right)\, dy.
\end{equation}
First, we note that $\Psi$ is well-defined around $\eta=0$ since
\begin{equation}
\begin{split}
 \lim_{y\to 0}y^{-1}\left(\left[\left(1+y\frac{u_x^0(\xi)}{\alpha}\right)^{-1}\right]^{\xi_2}_{\xi_1}\right)
 %&=\lim_{y\to 0}y^{\frac{2}{q}}\left[\sum_{n=1}^\infty{p+1 \choose n}\left[\left(y\frac{u_x^0(\xi_2)}{\alpha}\right)^n-\left(y\frac{u_x^0(\xi_1)}{\alpha}\right)^n\right]\right]^{-\frac{2}{p}}\\&=\lim_{y\to 0} \left[\sum_{n=1}^\infty{p+1 \choose n}y^{n-1}\left[\left(\frac{u_x^0(\xi_2)}{\alpha}\right)^n-\left(\frac{u_x^0(\xi_1)}{\alpha}\right)^n\right]\right]^{-\frac{2}{p}}\\&=\lim_{y\to 0} \left[\sum_{n=0}^\infty{p+1 \choose n+1}y^{n}\left[\left(\frac{u_x^0(\xi_2)}{\alpha}\right)^{n+1}-\left(\frac{u_x^0(\xi_1)}{\alpha}\right)^{n+1}\right]\right]^{-\frac{2}{p}}\\
&=\left(\frac{u_x^0(\xi_2)-u_x^0(\xi_1)}{|\alpha|}\right)<\infty.
\end{split}
\end{equation}
%which, for the choice of $\xi_1$ and $\xi_2$ according to \eqref{xi1xi2explicit}, is finite and positive.\\
%Since $p+1<0$ and $-\frac{2}{p}>0$, 
In fact, by partial fraction decomposition, we can write $\Psi$ explicitly as
\begin{equation}
\begin{split}
\Psi(\eta)&=\int_0^\eta y^{-1}\left(\left[\left(1+y\frac{u_x^0(\xi)}{\alpha}\right)^{-1}\right]^{\xi_2}_{\xi_1}\right)\, dy\\
&=\int_0^\eta\frac{1}{y}\frac{1+y\frac{u_x^0(\xi_1)}{\alpha}-\left(1+y\frac{u_x^0(\xi_2)}{\alpha}\right)}{\left(1+y\frac{u_x^0(\xi_1)}{\alpha}\right)\left(1+y\frac{u_x^0(\xi_2)}{\alpha}\right)} \,dy\\
&=\int_0^\eta \frac{u_x^0(\xi_2)-u_x^0(\xi_1)}{|\alpha|\left(1+y\frac{u_x^0(\xi_1)}{\alpha}\right)\left(1+y\frac{u_x^0(\xi_2)}{\alpha}\right)}\, dy\\
&=\int_0^\eta -\frac{u_x^0(\xi_2)}{\alpha\left(1+y\frac{u_x^0(\xi_2)}{\alpha}\right)}+\frac{u_x^0(\xi_1)}{\alpha\left(1+y\frac{u_x^0(\xi_1)}{\alpha}\right)}\, dy\\
&=\left[-\log\left(1+y\frac{u_x^0(\xi_2)}{\alpha}\right)+\log\left(1+y\frac{u_x^0(\xi_1)}{\alpha}\right)\right]_0^\eta\\
&=\log\left(\frac{1+\eta\frac{u_x^0(\xi_1)}{\alpha}}{1+\eta\frac{u_x^0(\xi_2)}{\alpha}}\right),
\end{split}
\end{equation}
which shows that $\Psi\to\infty$ as $\eta\to\eta*:=\frac{|\alpha|}{u_x^0(\xi_2)}$.\\
As $t\mapsto\eta(t)$ is the inverse of the monotonically increasing function $\Psi$, the minimal instance of blow-up for $\Psi$ is the smallest upper bound on $\eta$.
The instance of blow up for $\Psi$ is minimal for any choice of $\xi_2$ such that $u_x^0(\xi_2)=u_{max}$. Indeed, we have that
\begin{equation}
\min_{\xi_2\in[0,1]}\eta^*=\min_{\xi_2\in[0,1]}-\frac{\alpha}{u_x^0(\xi_2)}=\frac{|\alpha|}{u_{max}}.
\end{equation}
Therefore, the function $\eta$ is monotonically increasing and bounded from above by $\frac{|\alpha|}{u_{max}}$ and hence, the solution exists for all times. This proves the claim.\\

\textbf{The case $-1<a$ or $-\infty<\alpha<0$}\\
Applying Bernoulli's inequality to equation \eqref{eqeta} after squaring gives
\begin{equation}\label{inequa-10}
\begin{split}
1&=\left(\eta'(t)\right)^\frac{1}{a+1}\int_0^1\left[1+\eta(t)\frac{u^0_x(s)}{\alpha}\right]^{2\frac{\alpha}{2}}\,ds\\
&=\left(\eta'(t)\right)^\frac{1}{a+1}\int_0^1\left[1+2\eta(t)\frac{u^0_x(s)}{\alpha}+\frac{1}{\alpha^2}\eta^2(t)\Big(u_x^0(s)\Big)^2\right]^{\frac{\alpha}{2}}\,ds\\
&\geq\left(\eta'(t)\right)^\frac{1}{a+1}\int_0^1 1+\eta(t)u^0_x(s)+\frac{1}{2\alpha}\eta^2(t)\Big(u_x^0(s)\Big)^2\,ds\\
&=\left(\eta'(t)\right)^\frac{1}{a+1}\left(1+\frac{1}{2\alpha}\|u_x^0\|_{L^2(0,1)}^2\eta^2(t)\right),
\end{split}
\end{equation}
for all $t\in(0,t^*)$, where in the last step, we have also used Dirichlet or periodic boundary conditions. Since $a+1>0$, the function $x\mapsto x^{a+1}$ is monotonically increasing and it follows from \eqref{inequa-10} that $\eta$ can be bounded from above by a monotonically increasing function which goes to
\begin{equation}
\eta^*= \frac{\sqrt{2|\alpha|}}{\|u_x^0\|_{L^2(0,1)}},
\end{equation}
as $t\to\infty$. Now, we can estimate
\begin{equation}
\begin{split}
1+\eta(t)\frac{u^0_x(s)}{\alpha}&\geq 1-\eta(t)\frac{u_{max}}{|\alpha|}\\
&\geq 1-\frac{\sqrt{2|\alpha|}}{\|u_x^0\|_{L^2(0,1)}}\frac{u_{max}}{|\alpha|}\\
&>0,
\end{split}
\end{equation}
by assumption \eqref{Assa-10}. This proves the claim.\\

\textbf{The case $1\leq a$ or $-1\leq\alpha<0$}\\
Assume, to the contrary, that $1+\eta(t)\frac{u^0_x(\xi)}{\alpha}>0$, while $\eta'(t)>0$, for all $(\xi,t)\in[0,1]\times[0,\infty)$. Since $-1<\alpha<0$ for this parameter range, we can apply Jensen's inequality with the convex function $y\mapsto -y^{-\alpha}$ to equation \eqref{eqeta} to obtain
\begin{equation}
\begin{split}
1&=\left(\eta'(t)\right)^\frac{1}{a+1}\int_0^1\left[1+\eta(t)\frac{u^0_x(s)}{\alpha}\right]^{\alpha}\,ds\\
&\leq \left(\eta'(t)\right)^\frac{1}{a+1}\left(\int_0^1\left[1+\eta(t)\frac{u^0_x(s)}{\alpha}\right]^{-1}\,ds\right)^{-\alpha},\end{split}
\end{equation}
or, equivalently, after taking an $(a+1)$-power,
\begin{equation}\label{inequa1}
1\leq \eta'(t)\left(\int_0^1\left[1+\eta(t)\frac{u^0_x(s)}{\alpha}\right]^{-1}\,ds\right)^2.
\end{equation}

Integrating both sides of the inequality \eqref{inequa1} with respect to $t$ gives
\begin{equation}
t\leq \int_0^{\eta(t)}\left(\int_0^1\left[1+y\frac{u_x^0(s)}{\alpha}\right]^{-1}\right)^2\, dy\leq \int_0^{\frac{|\alpha|}{u_{max}}}\left(\int_0^1\left[1+y\frac{u_x^0(s)}{\alpha}\right]^{-1}\right)^2\, dy<\infty,
\end{equation}
which is a contraction. Here, we have used that $1+\eta(t)\frac{u_x^0(\xi)}{\alpha}>0$ for all $(\xi,t)\in [0,1]\times [0,\infty)$ for the second inequality. This proves the claim.

%Since, by assumption, the function on the right-hand side of \eqref{inequa1} is integrable over the interval $[0,\frac{|\alpha|}{u_{max}}]$, we find that $\eta$ can be bounded from below by a monotonically increasing function whose range includes the interval $[0,\frac{|\alpha|}{u_{max}}]$. Thus, for every $\varepsilon>0$, there exists a time $t^*(\varepsilon)$, such that, at the minimum of the expression $1+\eta (t)\frac{u^0_x(\xi)}{\alpha}$ in $\xi$, we find that\begin{equation}\begin{split}1-\eta(t^*)\frac{u_{max}}{|\alpha|}&\leq 1-\frac{u_{max}}{|\alpha|}\left(\frac{|\alpha|}{u_{max}}-\varepsilon\right)\\&<\frac{u_{max}}{|\alpha|}\varepsilon,\end{split}\end{equation}which, since $\varepsilon$ was arbitrary, contradicts our assumption. This proves the claim.
\end{proof}

\begin{remark}
Generally, the integral expression on the left-hand side of equation \eqref{eqeta} can be expanded in a Binomial series around $\eta=0$ as 
\begin{equation}
\int_0^1\left[1+\eta\frac{u^0_x(s)}{\alpha}\right]^{\alpha}\,ds=\sum_{n=0}^\infty {\alpha\choose n} \alpha^{-n}\left[\int_0^1\Big(u_x^0(s)\Big)^n\, ds\right]\eta^n,
\end{equation}
which, since
\begin{equation}
\lim_{n\to\infty}\left|{\alpha \choose n}\right|^{\frac{1}{n}}=\lim_{n\to\infty}\frac{1}{\Gamma(-\alpha)^{\frac{1}{n}}n^{\frac{\alpha+1}{n}}}=1,
\end{equation}
for $\alpha\notin \mathbb{N}$ (where the Gamma function has poles), defines an analytic function with radius of convergence 
\begin{equation}\label{radiusofconvergence}
\begin{split}
R&=|\alpha|\left\{\limsup_{n\to\infty}\left|\int_0^1\Big(u_x^0(s)\Big)^n\, ds\right|^{\frac{1}{n}}\right\}^{-1}\\
&\geq|\alpha|\|u_x^0\|_{L^\infty(0,1)},
\end{split}
\end{equation}
where the inequality in \eqref{radiusofconvergence} can be deduced by passing to the subsequence of even $n's$.
Since $u_{max}$ and $u_{min}$ can be very different, we see that there is no hope on a general global existence or singularity criterion solely relying upon the failure of analyticity, i.e., of $\eta$ approaching $\|u_x^0\|_{L^\infty(0,1)}$.
\end{remark}

The following theorem presents more refined conditions which guarantee global existence in the parameter range $a>-1$.

\begin{theorem}[Improved Global Existence for $a>-1$]\label{improvedthm}
Let $u\in C\Big(0,t^*; C^{1,Lip}(0,1)\Big)$ be a solution to equation \eqref{GPJa} with $-1<a$, either with Dirichlet boundary conditions \eqref{DirichletBC} or with periodic boundary conditions and zero mean \eqref{periodicBC}, and let  $t^*\in (0,\infty]$ be its maximal existence time. 
\begin{itemize}
	\item For any $n\geq2$, let $\eta^*$ be the minimal positive root (if such a root exists) of the polynomial
	\begin{equation}
	p_n(\eta)=1+\frac{1}{n}\sum_{k=2}^n{n \choose k}\left(\frac{-2}{a+1}\right)^{1-k}\left[\int_0^1\Big(u_x^0(s)\Big)^k\, ds\right]\eta^k,
	\end{equation}
	and assume that
	\begin{equation}
	\eta^*u_{max}<\frac{2}{a+1}.
	\end{equation}
	Then, the solution exists globally in time, i.e., $t^*=\infty$.
\end{itemize}
\end{theorem}
\begin{proof}
The proof is completely analogous to the proof of the third case in Theorem \ref{mainthm}. Indeed, we can apply Bernoulli's inequality after taking an $n^{th}$ power in the integral of equation \eqref{eqeta} to obtain
\begin{equation}
\begin{split}
1&=\left(\eta'(t)\right)^\frac{1}{a+1}\int_0^1\left[1+\eta(t)\frac{u^0_x(s)}{\alpha}\right]^{n\frac{\alpha}{n}}\,ds\\
&\geq \left(\eta'(t)\right)^\frac{1}{a+1}\left(1+\frac{1}{n}\sum_{k=2}^n{n \choose k}\alpha^{1-k}\left[\int_0^1\Big(u_x^0(s)\Big)^k\, ds\right]\eta^k\right),
\end{split}
\end{equation}
implying, since $x\mapsto x^{a+1}$ is monotonically increasing, that $\eta$ can be bounded from above by a monotonically increasing function which can be bounded from above by $\eta^*$. Since, by assumption $\eta^*u_{max}<-\alpha$, the claim readily follows. 
\end{proof}

%\textbf{The case $0<a<3$ or $-2<\alpha<-\frac{1}{2}$}\\The proof is similar to the previous parameter range. Applying Bernoulli's inequality to equation \eqref{eqeta} after taking a third power gives\begin{equation}\label{inequa03}\begin{split}1&=\left(\eta'(t)\right)^\frac{1}{a+1}\int_0^1\left[1+\eta(t)\frac{u^0_x(s)}{\alpha}\right]^{3\frac{\alpha}{3}}\,ds\\&=\left(\eta'(t)\right)^\frac{1}{a+1}\int_0^1\left[1+3\eta(t)\frac{u^0_x(s)}{\alpha}+3\frac{1}{\alpha^2}\eta^2(t)\Big(u_x^0(s)\Big)^2+\frac{1}{\alpha^3}\eta^3(t)\Big(u_x^0(s)\Big)^3\right]^{\frac{\alpha}{3}}\,ds\\&\geq\left(\eta'(t)\right)^\frac{1}{a+1}\int_0^1 1+\eta(t)u^0_x(s)+\frac{1}{\alpha}\eta^2(t)\Big(u_x^0(s)\Big)^2+\frac{1}{3\alpha^2}\eta^3(t)\Big(u_x^0(s)\Big)^3\,ds\\&=\left(\eta'(t)\right)^\frac{1}{a+1}\left(1+\frac{1}{\alpha}\|u_x^0\|_{L^2(0,1)}^2\eta^2(t)+\frac{1}{3\alpha^2}\left[\int_0^1\Big(u_x^0(s)\Big)^3\,ds\right]\eta^3(t)\right),\end{split}\end{equation}for all $t\in(0,t^*)$, where in the last step, we have also used Dirichlet or periodic boundary conditions. Now, $\eta$ can be bounded from below by a monotonically increasing function whose range is bounded from above by the minimal positive root of the cubic polynomial\begin{equation}p(y)=1+\frac{1}{\alpha}\|u_x^0\|_{L^2(0,1)}^2y^2+\frac{1}{3\alpha^2}\left[\int_0^1\Big(u_x^0(s)\Big)^3\,ds\right]y^3,\end{equation}which we denote as $y^*$. By assumption, $y^*<\frac{|\alpha|}{u_{max}}$, thus implying that $1+\frac{\eta(t)}{\alpha}u_x^0(\xi)>1-\frac{y^*u_{max}}{|\alpha|}>0$ for all $t\in[0,\infty)$. This proves the claim.

\begin{remark}
If, in addition to the assumptions of Theorem \ref{improvedthm}, the initial velocity field $u^0$ is even, i.e.,
\begin{equation}
u^0\left(\frac{1}{2}-x\right)=u^0\left(\frac{1}{2}+x\right),\quad x\in\left[0,\frac{1}{2}\right],
\end{equation}
or odd, i.e.,
\begin{equation}
u^0\left(\frac{1}{2}-x\right)=-u^0\left(\frac{1}{2}+x\right),\quad x\in\left[0,\frac{1}{2}\right],
\end{equation}
then less restrictive assumptions can be made to guarantee global existence.
\end{remark}

\begin{example}\label{example}
Let us consider the initial condition $u_0(x)=\gamma x(x-1)$, for which finite-time blow up was proved in \cite{childress_ierley_spiegel_young_1989}, once again and contrast it to assumption \eqref{Assa-10}. We calculate
\begin{equation}
\begin{split}
\frac{1}{\sqrt{a+1}}\|u_x^0\|_{L^2(0,1)}&=\sqrt{\frac{1}{2}\int_0^1\gamma^2(2x-1)^2\,dx}\\
&=\frac{\gamma}{\sqrt{3}},
\end{split}
\end{equation}
while $u_{max}=\gamma$, which shows that the condition \eqref{Assa-10} is not satisfied for this initial condition.
\end{example}
\begin{example}[$a\geq1$, Proudman--Johnson-type singularity]
Consider the GPJ equation with $a\geq1$. In their paper \cite{childress_ierley_spiegel_young_1989}, the authors analyzed different blow-up scenarios for different initial conditions, cf. also Example \ref{examplea=1}. Here, we consider the initial condition
\begin{equation}\label{ua=1}
u_x^0(x)=
\begin{cases}\gamma-4\gamma x\qquad&\text{ for } 0<x<\frac{1}{2}\\
4\gamma x-3\gamma\qquad&\text{ for } \frac{1}{2}<x<1,
\end{cases}
\end{equation}
for $\gamma\in\mathbb{R}$, such that $u^0$ satisfies either Dirichlet boundary conditions or is mean-free. The Fourier series of $u_x^0$ is given by
\begin{equation}\label{Fourierexample}
u_x^0(x)=\frac{\gamma}{2\pi^2}\sum_{n\in\mathbb{Z}^*}\frac{(e^{\pi\ri n}-1)[2(e^{\pi\ri n}-1)-\pi\ri n(e^{\pi\ri n}+1)]}{n^2}e^{2\pi\ri nx},
\end{equation}
and since $u_x^0$ is Lipschitz continuous, we find that the Fourier series \eqref{Fourierexample} converges uniformly, cf. Figure \ref{FigAbsandmode3}. Hence, we can choose a sufficiently high truncation of \eqref{Fourierexample} to obtain a smooth (even analytic) initial condition that satisfies the singularity criterion in Theorem \ref{mainthm}.\\
\begin{figure}
	\includegraphics[scale=0.6]{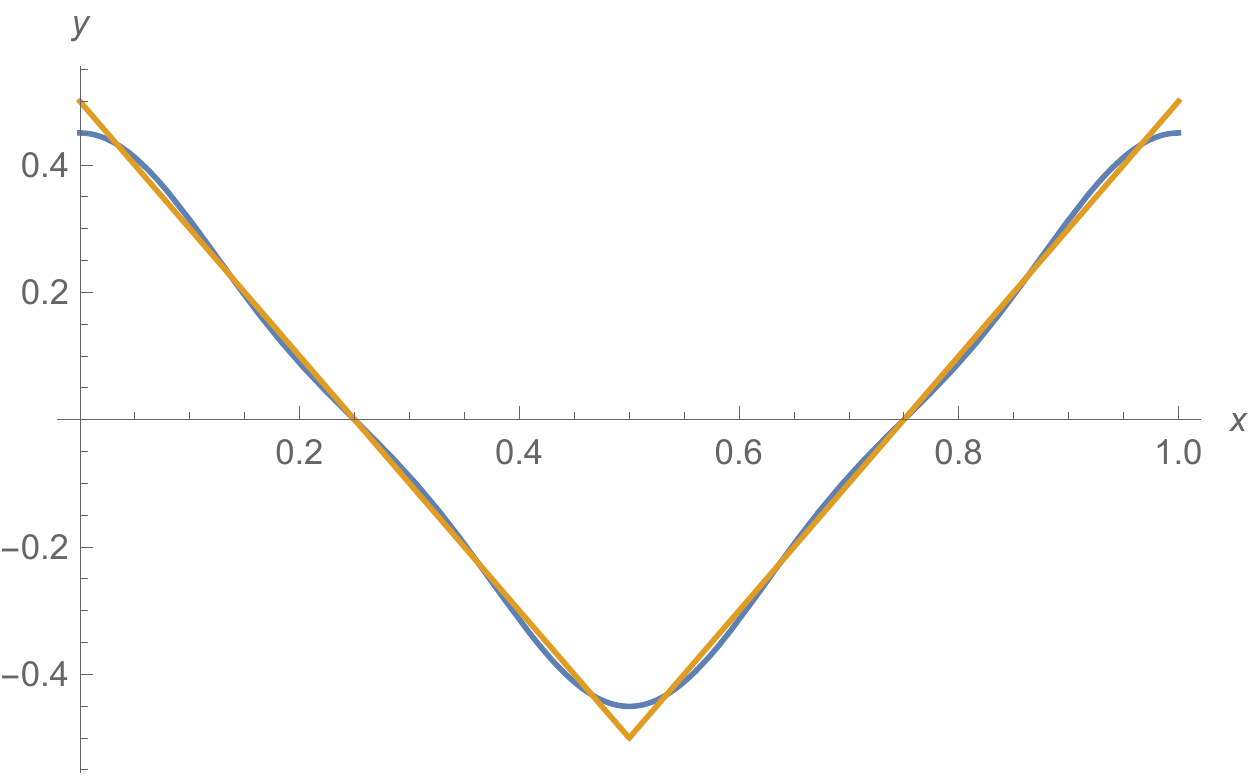}
	\caption{Initial condition \eqref{ua=1} for $\gamma=0.5$ together with its three-mode truncation.} 
	\label{FigAbsandmode3}
\end{figure}
%Indeed, we calculate\begin{equation}\begin{split}\int_0^1\Big(u_x^0(s)\Big)^n\, ds&=\gamma^n\left[\int_0^\frac{1}{2}(1-4s)^n\,ds+\int_{\frac{1}{2}}^1(4s-3)^n\,ds\right]\\&=\gamma^n\sum_{k=0}^n{n\choose k}\frac{1}{k+1}\Big[(-4)^k2^{-k-1}+(-3)^{n-k}4^k(1-2^{-k-1})\Big]\\&=\frac{\gamma^n}{n+1}\sum_{k=0}^n{n+1\choose k+1}2^{k-1}\Big[(-1)^k+(-3)^{n-k}(2^{k+1}-1)\Big]\\&=\frac{\gamma^n}{4(n+1)}\sum_{k=1}^{n+1}{n+1\choose k}2^k\Big[(-1)^{k-1}+(-3)^{n-k+1}(2^k-1)\Big]\\&=\frac{\gamma^n}{4(n+1)}\left[1-(-1)^{n+1}+1-(-1)^{n+1}\right]\\&=\frac{\gamma^n}{2(n+1)}\left[1-(-1)^{n+1}\right],\end{split}\end{equation}which implies that\begin{equation}\limsup_{n\to\infty}\left|\int_0^1\Big(u_x^0(s)\Big)^n\, ds\right|^{\frac{1}{n}}=\gamma,\end{equation}and hence, condition \eqref{limsupcond} for \eqref{ua=1} becomes\begin{equation}|\gamma|<1.\end{equation}
Indeed, equation \eqref{eqeta} can be integrated explicitly to
\begin{equation}\label{intexample}
\begin{split}1&=\sqrt{\eta'(t)}\int_0^1\frac{1}{1-\frac{a+1}{2}\eta(t)u^0_x(s)}\,ds\\
&=\sqrt{\eta'(t)}\left(\int_0^{\frac{1}{2}}\frac{ds}{1-\frac{a+1}{2}\gamma\eta(t)+2(a+1)\gamma\eta(t)s}+\int_{\frac{1}{2}}^1\frac{ds}{1+3\frac{a+1}{2}\gamma\eta(t)-2(a+1)\gamma\eta(t)s}\right)\\
&=\sqrt{\eta'(t)}\left(\frac{1}{2(a+1)\gamma \eta(t)}\left[\log\left(1-\frac{a+1}{2}\gamma\eta(t)+2(a+1)\gamma\eta(t)s\right)\right]_{0}^{\frac{1}{2}}\right.\\
&\qquad\qquad\qquad\qquad\left.-\frac{1}{2(a+1)\gamma \eta(t)}\left[\log\left(1+3\frac{a+1}{2}\gamma\eta(t)-2(a+1)\gamma\eta(t)s\right)\right]_{\frac{1}{2}}^1\right)\\
&=\frac{\sqrt{\eta'(t)}}{2(a+1)\gamma \eta(t)}\left[\log\left(1+\frac{a+1}{2}\gamma \eta(t)\right)-\log\left(1-\frac{a+1}{2}\gamma \eta(t)\right)\right.\\
&\qquad\qquad\qquad\qquad\left.-\log\left(1-\frac{a+1}{2}\gamma \eta(t)\right)+\log\left(1+\frac{a+1}{2}\gamma \eta(t)\right)\right]\\
&=\frac{\sqrt{\eta'(t)}}{2(a+1)\gamma\eta(t)}\log\left|\frac{1+\frac{a+1}{2}\gamma\eta(t)}{1-\frac{a+1}{2}\gamma\eta(t)}\right|,
\end{split}
\end{equation}
and, since the logarithmic expression in \eqref{intexample} is square-integrable, it follows that $\eta$ becomes unbounded and hence, the solution becomes singular in finite time. The blow-up time depends on $a$ and $\gamma$.\\
%Because $u_x^0$ as defined in \eqref{ua=1} is Lipschitz continuous, its Fourier series \eqref{Fourierexample} converges uniformly on the interval $[0,1]$ and hence, 
For a sufficiently high modular truncation, we have also found a smooth (in fact real analytic) initial condition which leads to formation of a singularity for the parameter value $a>1$. This answers a question by Okamoto \cite{Okamotopresentation}.\\
Note, however, that for the one-mode truncation
\begin{equation}\label{initialcos}
u_x^0=\frac{8\gamma}{\pi^2}\cos(2\pi x),
\end{equation}
equation \eqref{eqeta} takes the form
\begin{equation}\label{intexample2}
\begin{split}1&=\sqrt{\eta'(t)}\int_0^1\frac{1}{1-\eta(t)u^0_x(s)}\,ds\\
&=\sqrt{\eta'(t)}\int_0^1\frac{ds}{1-\eta(t)\frac{8\gamma}{\pi^2}\cos(2\pi s)}\\
&=\sqrt{\eta'(t)}\frac{1}{\pi}\int_0^\infty\frac{ds}{1-\eta(t)\frac{8\gamma}{\pi^2}+\left(1+\eta(t)\frac{8\gamma}{\pi^2}\right)s^2}\\
&=\sqrt{\eta'(t)}\frac{1}{1-\eta(t)\frac{8\gamma}{\pi^2}},
\end{split}
\end{equation}
and hence, $\eta(t)< \frac{8\gamma}{\pi^2}$ for all $t\in(0,\infty)$. In particular, the solution with initial condition \eqref{initialcos} exists for all times.\\
Also, we can check condition \eqref{Assa-10} for the initial condition \eqref{initialcos}. Indeed, we find that
\begin{equation}
u_{max}=\frac{8\gamma}{\pi^2}>\frac{8\gamma}{\pi^2}\frac{1}{\sqrt{2}\sqrt{a+1}}=\frac{1}{\sqrt{a+1}}\|u_x^0\|_{L^2(0,1)},
\end{equation}
which shows that condition \eqref{Assa-10} is violated for $a$ sufficiently large.
\end{example}

\section{Summary and Further Perspectives}
We proved finite-time existence criteria as well as global existence criteria for the GPJ equation in dependence of the parameter $a$. We strengthened existing conditions on the formation of a singularity and provided new conditions depending on properties of the initial velocity profile. In particular, we showed that for $a>1$  solutions with sufficiently well behave initial conditions will exist for all times, while sufficiently singular - but still smooth - initial conditions will lead to the formation of a finite-time singularity. This behavior is exemplified for a truncation of a zigzag function. In the appendix, we also gave a physical derivation of the GPJ equation from the compressible, two-dimensional Euler equations.\\
While the finite-time existence results in, e.g., \cite{Sarria2013, SarriaSaxton2013} also specified the nature of the singularity, such as blow-up in the $L^p$-norm of $u_x$, Theorem \ref{mainthm} only guarantees that the solution cannot be extended, but does not give further information on the nature of the singularity. Does the solution blow up or does the solution lose its (Lipschitz) regularity, i.e., is there blow-up in its derivatives? In \cite{Sarria2013}, the authors showed, for $a=1$, that there exist smooth initial data such that the solutions exist globally. Since their results were obtained with quite different techniques, it would be interesting to compare the singularity and global existences results in Theorem \ref{mainthm} for $a\geq 1$ and $a>-1$ respectively to the case of $a=1$ in \cite{Sarria2013}\\
Also, it would be interesting to extend the flow map approach laid out in the present paper to also give information on how a solution becomes singular. This will, however, also require information on the second derivative of $t\mapsto\eta(t)$, which has to be derived from new estimates.\\

\textbf{Acknowledgments}. \\
	The author would like to thank the anonymous reviewer for several useful comments and suggestions.

\section{Appendix: Physical Derivation of the GPJ Equation}\label{Appendix}
In this section, we provide a physical derivation of the GPJ equation for a general value of $a$. For special values of $a$, the GPJ equation is already known to have physical applications, cf. \cite{okamoto2000}. We will find that the GPJ equations model the dynamics of an inviscid, compressible fluid close to a wall. This is an immediate generalization of the initial motivation of the Proudman--Johnson equation, modeling inviscid, incompressible fluid motion close a wall, cf. \cite{childress_ierley_spiegel_young_1989}.\\
Consider the three-dimensional Euler equations for an isentropic, compressible fluid in vorticity form, \cite[p.24]{chorin2012mathematical},
\begin{equation}\label{Euler}
\frac{\partial}{\partial t}\left(\frac{\boldsymbol{\omega}}{\rho}\right)+(\mathbf{u}\cdot\nabla)\left(\frac{\boldsymbol{\omega}}{\rho}\right)=\left(\frac{\boldsymbol{\omega}}{\rho}\cdot\nabla\right)\mathbf{u},
\end{equation}
together with the equation of mass conservation
\begin{equation}
\frac{\partial\rho}{\partial t}+\nabla\cdot(\rho\mathbf{u})=0,
\end{equation}
for a three-dimensional velocity field $\mathbf{u}:\mathbb{R}^3\times (0,t^*)\to\mathbb{R}^3,\quad (x,y,z:t)\mapsto\mathbf{u}(x,y,z;t)$, the three-dimensional vorticity $\boldsymbol{\omega}=\nabla\times\mathbf{u}$ and the non-negative mass density $\rho:\mathbb{R}^3\times(0,t^*)\to\mathbb{R}$. Here, we denoted the minimum of the maximal existence times of $\mathbf{u}$ and $\rho$ as $t^*$.\\
After differentiating the product expressions, equation \eqref{Euler} can be written equivalently in terms of the material derivative $\frac{D}{Dt}:=\frac{\partial}{\partial t}+\mathbf{u}\cdot\nabla$
as 
\begin{equation}\label{Eulermaterial}
\frac{1}{\rho}\frac{D\boldsymbol{\omega}}{Dt}-\frac{1}{\rho^2}\frac{D\rho}{Dt}\boldsymbol{\omega}=\left(\frac{\boldsymbol{\omega}}{\rho}\cdot\nabla\right)\mathbf{u}.
\end{equation}
Since the mass density $\rho$ is non-negative, there exists a scalar function $R:\mathbb{R}^3\times(0,t^*)\to\mathbb{R}$ such that
\begin{equation}\label{exprho}
\rho(x,y,z;t)=e^{R(x,y,z;t)}.
\end{equation}
Assuming a velocity field of the form
\begin{equation}\label{ansatzu}
\mathbf{u}(x,y,z;t)=\Big(u(x,t),\beta u_x(x,t)y,0\Big),
\end{equation}
for $\beta\in\mathbb{R}$ a parameter, an unknown scalar function $u:\mathbb{R}\times(0,t^*)\to\mathbb{R}$ and $y\geq 0$, the three-dimensional vorticity becomes
\begin{equation}
\boldsymbol{\omega}(x,y,z;t)=\Big(0,0,-\beta u_{xx}(x,t)y\Big),
\end{equation}
while the divergence of $\mathbf{u}$ reads
\begin{equation}\label{divu}
\nabla\cdot\mathbf{u}=(1+\beta)u_x.
\end{equation}
Introducing the scalar vorticity
\begin{equation}
\omega(x,y;t)=-\beta u_{xx}(x,t)y,
\end{equation}
equation \eqref{Eulermaterial} simplifies to the scalar equation
\begin{equation}
\frac{1}{\rho}\frac{D\omega}{Dt}-\frac{1}{\rho^2}\frac{D\rho}{Dt}\omega=0,
\end{equation}
or equivalently, in terms of the exponential density \eqref{exprho}, after multiplication by $e^R$,
\begin{equation}\label{Eulerred}
\frac{D\omega}{Dt}=\frac{DR}{Dt}\omega.
\end{equation}
Since, by \eqref{exprho} and by \eqref{divu}, the equation of mass conservation becomes
\begin{equation}
\frac{DR}{Dt}=-\nabla\cdot\mathbf{u}=-(1+\beta)u_x,
\end{equation}
the Euler equations \eqref{Eulerred} take the form
\begin{equation}
\omega_t+u\omega_x+\beta u_x\omega_y=-(1+\beta)u_x\omega,
\end{equation}
or, in terms of the scalar velocity $u$, after division by $-\beta y$,
\begin{equation}\label{GPJxx}
u_{xxt}+uu_{xxx}+(1+2\beta)u_xu_{xx}=0.
\end{equation}
Setting
\begin{equation}
a:=-(1+2\beta),
\end{equation}
we obtain the GPJ equation 
\begin{equation}\label{GPJapendix}
u_{xxt}+uu_{xxx}-au_xu_{xx}=0,
\end{equation}
as defined in \eqref{GPJa}.\\
If we now assume periodic boundary conditions in the ansatz \eqref{ansatzu}, or if we define the scalar velocity $u$ on the interval $[0,1]$ together with Dirichlet boundary conditions, we obtain the according boundary conditions for the GPJ equation.\\
As the velocity field \eqref{ansatzu} becomes unbounded for $y\to\infty$, the GPJ equations give a model for, e.g., near-wall dynamics for compressible, inviscid fluids only for small enough $y$. \\

\bibliographystyle{abbrv}
\bibliography{/Users/floriankogelbauer/Dropbox/Bibs/PDE}
\end{document}